\documentclass[12pt]{amsart} 

\usepackage{amscd} 
\usepackage{amsmath} 
\usepackage{amssymb} 
\usepackage{amsthm}
\usepackage{bigdelim}
\usepackage{color} 
\RequirePackage[dvipsnames,usenames]{xcolor}\usepackage{enumerate}
\usepackage{mathrsfs}
\usepackage{multirow}
\usepackage[all]{xy} 
\usepackage{hyperref}
\usepackage{soul}
\newtheorem{thm}{Theorem}[section] 
\newtheorem{cor}[thm]{Corollary}
\newtheorem{prop}[thm]{Proposition}

\newtheorem{lem}[thm]{Lemma}
 
\theoremstyle{definition} 
\newtheorem{defn}[thm]{Definition}

\newtheorem{eg}[thm]{Example} 
\theoremstyle{remark}
\newtheorem{rem}[thm]{Remark}

\newtheorem{notation}[thm]{Notation}

\newtheorem{step}{Step}[thm]

\newtheorem*{ack}{Acknowledgements}

\usepackage{marginnote}

\DeclareMathOperator{\coker}{coker}
\DeclareMathOperator{\imSat}{imSat}
\DeclareMathOperator{\Tr}{Tr}

\DeclareMathOperator{\Spec}{Spec}
\DeclareMathOperator{\Isom}{Isom}

\newcommand{\factor}[2]{\left. \raise 2pt\hbox{\ensuremath{#1}} \right/
        \hskip -2pt\raise -2pt\hbox{\ensuremath{#2}}}

\newcommand{\factormiddle}[2]{\left. \raise 4pt\hbox{\ensuremath{#1}} \right/
        \hskip -2pt\raise -4pt\hbox{\ensuremath{#2}}}

\newcommand{\factorbig}[2]{\left. \raise 6pt\hbox{\ensuremath{#1}} \right/
        \hskip -2pt\raise -6pt\hbox{\ensuremath{#2}}}

\newcommand{\factorBig}[2]{\left. \raise 8pt\hbox{\ensuremath{#1}} \right/
        \hskip -2pt\raise -8pt\hbox{\ensuremath{#2}}}

\newcommand{\sF}{\mathcal{F}}
\newcommand{\sG}{\mathcal{G}}
\newcommand{\sH}{\mathcal{H}}
\newcommand{\sO}{\mathcal{O}}
\newcommand{\sT}{\mathcal{T}}

\newcommand{\bQ}{\mathbb{Q}}

\hypersetup{
bookmarks,
bookmarksdepth=3,
bookmarksopen,
bookmarksnumbered,
pdfstartview=FitH,
colorlinks,backref,hyperindex,
linkcolor=Sepia,
anchorcolor=BurntOrange,
citecolor=MidnightBlue,
citecolor=OliveGreen,
filecolor=BlueViolet,
menucolor=Yellow,
urlcolor=OliveGreen
}

\title{The Demailly--Peternell--Schneider conjecture is true in positive characteristic}
\author{Sho Ejiri}
\address{Department of Mathematics, Graduate School of Science, Osaka Metropolitan University, Osaka City, Osaka 558-8585, Japan}
\email{shoejiri.math@gmail.com}
\author{Zsolt Patakfalvi}
\address{EPFL SB MATH CAG 
MA C3 635 
Station 8, 
CH-1015, Lausanne}
\email{zsolt.patakfalvi@epfl.ch}
\begin{document}
\maketitle
\markboth{SHO EJIRI AND ZSOLT PATAKFALVI}{The Demailly--Peternell--Schneider conjecture is true in positive characteristic}
\begin{abstract}
We prove the Demailly--Peternell--Schneider conjecture in positive characteristic: if $X$ is a smooth projective  variety over an algebraically closed field of characteristic $p>0$ with  $-K_X$ is nef, then the  Albanese morphism $a: X \to A$ is surjective. We also show strengthenings either allowing mild singularities for $X$, or proving more special properties of $a$.

The above statement for compact K\"ahler manifolds was conjectured originally by Demailly, Peternell and Schneider in 1993, and for smooth projective varieties of characteristic  zero it was shown by Zhang in 1996. In positive characteristic, all earlier results involved tameness assumptions either on cohomology or on the singularities of the general fibers of $a$. The main feature of the present article is the development of a technology  to avoid such assumptions.   
\end{abstract}
\section{Introduction}
\emph{Throughout this paper, we work over an algebraically closed field $k$ of positive characteristic.} 

There has been many works in multiple branches of geometry saying that a ``uniformly non-hyperbolic'' compact geometric space $X$ has very special Albanese morphism $a : X \to A$. Here, the word ``uniformly'' is a synonym of  ``in all directions''. 

The starting points of such statements are the following  basic principles:
\begin{itemize}
\item images of ``uniformly non-hyperbolic'' spaces are ``non-hyperbolic'', and 
\item subspaces of abelian varieties are ``hyperbolic'', unless they are abelian sub-varieties (or equivalently linear, in the analytic use of language). 
\end{itemize}
For example, using that $a(X)$ is not an abelian subvariety of $A$, the above two basic principles imply the contradiction that $a(X)$ is both hyperbolic and non-hyperbolic at the same time, unless $a : X \to A$ is surjective.  

To turn the above general philosophy on the special structure of the Albanese morphism $a: X \to A$ into a precise theorem or conjecture, one has to fix:
\begin{itemize}
\item the branch of geometry (the citation are only examples of an enormous literature on the subject): most typically this would be complex geometry with all different degrees of analytic or K\"ahler flavor  \cite{Bog74,Kaw81,Bea83,CPZ03,Dem15,Cao15,CH17,Mat22}, or characteristic zero 
\cite{Zha96,Zha05,CH01,LTZZ,Wan22}  or positive characteristic algebraic geometry \cite{HP16a,HPZ19,Eji19p,Eji19w,Wan22,BJ23}, and the theory of foliations also has variants \cite{Dru21},
\item the meaning of ``uniformly non-hyperbolic'' (again only examples are cited): this could mean a positive curvature condition \cite{Dem15}, $K_X \equiv 0$ \cite{Bea83}, $-K_X$ is nef \cite{Dem15,Eji19w}, $\kappa(X)\leq 0$, $\kappa_S(X) \leq 0$ \cite{HPZ19}, $\kappa(X)=0$ \cite{CH01}, or $\kappa_S(X)=0$ \cite{HP16a}. 
\item the structure one asks/shows for $a: X \to A$: typically to be as close as one can get to a locally trivial fiber bundle, which then might need to be relaxed based on the setting. For example one might need to allow a birational transformation. Or, sometimes the best one can show is that $a$ is surjective. 
\end{itemize}
Our main theorem focuses on the positive characteristic case, where a special issue comes up: the general fiber of $a$ can be very singular. Its singularities are homeomorphic to a smooth point, but the algebraic structure can be very different compared to a smooth point. This breaks down most techniques, and hence all past results were assuming the existence of general fibers with good singularities, e.g., \cite{Eji19p,Eji19w}, or some ordinarity assumption, which secretely also assumes good general fibers \cite{HP16a,HPZ19,BJ23}.

\subsection{Smooth case of our main result.}
\label{sec:smooth_case}

Demailly, Peternell and Schneider  conjectured that a compact K\"ahler manifold $X$ with $-K_X$ nef has surjective Albanese \cite[Conj 2]{DPT93}. 
Our main result is the positive characteristic version of this conjecture. 
In particular, our main theorem does not involve any kind of assumptions explained in the previous paragraph, which is especially hard to attain in positive characteristic.

\begin{thm}[\scshape Smooth case of Theorem \ref{thm:main}] \label{thm:main_smooth} 
If $X$ is a smooth projective variety with $-K_X$ nef, then the Albanese morphism $a :X \to A $ is surjective.
\end{thm}

 Theorem \ref{thm:main_smooth} was  shown in characteristic zero by Zhang in \cite[Thm 1]{Zha96}. Hence, Theorem \ref{thm:main_smooth} now holds over any algebraically closed field. In fact, using the compatibility of the Albanese morphism with base-change, it even holds over non-closed fields.  
 
 For further references on earlier results, in positive characteristic, in characteristic zero, and also in the analytic setting we refer to the part of the Introduction preceding Section \ref{sec:smooth_case}.
 
 Theorem \ref{thm:main_smooth} also yields a corollary about the characterization of abelian varieties:
 
 \begin{cor}
 \label{cor:characterization_abelian}
 Let $X$ be a smooth projective variety with $-K_X$ nef, and set $b_1= \dim_{\bQ_l} H^1_{\textrm{\'et}} (X, \bQ_l)$.  Then:
 \begin{enumerate}
 \item \label{itm:characterization_abelian:inequality} $b_1 \leq 2 \dim X$, and 
 \item \label{itm:characterization_abelian:characterization} $b_1 = 2 \dim X$ if and only if $X$ is an abelian variety. 
 \end{enumerate}
 \end{cor}

\subsection{Singularities and the structure of the Albanese}
The general case of our main result is the following. This version  allows singularities of $X$,  a boundary divisor, and it also discusses further structure of the Albanese morphism:

\begin{thm} \label{thm:main}
Let $(X, \Delta)$ be a projective, strongly $F$-regular pair, for which there exists an integer $i>0$ not divisible by $p$ such that $-i(K_X+\Delta)$ is  nef  and Cartier. 
Then the following hold:
\begin{enumerate}
\item {\scshape Surjectivity.} \label{itm:main:Albanese_surjective} The Albanese morphism $a:X\to A$ of $X$ is surjective. 
\item {\scshape Weak flatness.} \label{itm:main:a_exc} There are no $a$-exceptional divisors, that is, for every prime divisor $E$ on $X$, the codimension of $f(\mathrm{Supp}(E))$ is at most one.  
\item {\scshape Fibration up to universal homeomorphism.} \label{itm:main:Stein} Let $a:X\xrightarrow{f} Y \xrightarrow{g} A$ be the Stein factorization of $a$. Then $g$ is purely inseparable. 
\item {\scshape Fibration.} \label{itm:main:fiber_space} If $f$ is separable, then $a$ forms an algebraic fiber space, i.e., $g$ is an isomorphism. 
\end{enumerate}
\end{thm}
To be on the safe side, let us mention that the divisor in a pair is always \emph{effective} in this article, and a Cartier divisor has always $\mathbb{Z}$-coefficients. Hence, the $i(K_X+\Delta)$ being Cartier in Theorem \ref{thm:main} implies that $\Delta$ is a $\mathbb{Z}_{(p)}$-divisor. 

As we explained above,  the main feature of Theorem \ref{thm:main} is that it deals with arbitrarily singular general fibers $G$ of the Albanese. In fact, 
Theorem~\ref{thm:main} has  been shown when $(G,\Delta|_G)$ has $F$-pure singularities \cite[Theorem~1.1]{Eji19w}. Examples of  variety to which Theorem \ref{thm:main}  (or even Theorem \ref{thm:main_smooth}) pertains, but the earlier results do not, are quasi-hyperelliptic surfaces, or the varieties with $-K_X$ nef and $G$ non-reduced in \cite[Section 14]{PZ19}. See, Remark \autoref{rem:alb} for more on this. 
%
%
%
%

We note that one can improve a little the statements of Theorem \ref{thm:main} in some special cases:
\begin{cor} \label{cor:surface}
Let $(X,\Delta)$ be a strongly $F$-regular projective surface pair such that 
 $-i(K_X+\Delta)$ is 
a nef Cartier divisor for an integer $i>0$ not divisible by $p$. 
Then the Albanese morphism of $X$ forms an algebraic fiber space. 
\end{cor}
\begin{cor} \label{cor:K=0}
Assume that $p\ge 5$. 
Let $X$ be a normal projective variety with strongly $F$-regular singularities 
such that $iK_X$ is a numerically trivial Cartier divisor 
for an integer $i>0$ not divisible by $p$.
Let $a:X\to A$ be the Albanese morphism of $X$ and 
let $a:X \to Y \to A$ be the Stein factorization of $a$. 
If $\dim Y = \dim X-1$, then $a$ forms an algebraic fiber space. 
\end{cor}
\subsection{Local triviality.}
Having shown Thm \ref{thm:main}, one might wonder if $f$ is actually trivial after an adequate base-change. Such statements are known in characteristic zero, e.g.,  \cite{Cao16} or \cite[Appendix]{PZ19}. It turns out that the answer is also  yes in our situation, if the base is a curve. However, one can take the base-change to be finite only over $\overline{\mathbb F}_p$. 
\begin{thm} \label{thm:isotrivial}
Let $(X,\Delta)$ be a projective strongly $F$-regular pair with $K_X+\Delta$ is $\mathbb Z_{(p)}$-Cartier, and let $a:X \to A$ be the Albanese morphism of $X$ with $\dim a(X)=1$. Further, assume either that 
\begin{enumerate}
\item \label{itm:isotrivial:semi_ample} $-(K_X+\Delta)$ is semi-ample or 
\item \label{itm:isotrivial:finite_field} $-(K_X+\Delta)$  is nef and $k=\overline{\mathbb F}_p$. 
\end{enumerate}
Then $a:X\to A $ is an algebraic fiber space (in particular $A$ is an elliptic curve) and the pairs $(F, \Delta|_F)$ for different choices of closed fibers $F$ are isomorphic. 

Additionally, if $k= \overline{\mathbb F}_{p}$ (case \eqref{itm:isotrivial:finite_field}), then
there exists a finite morphism $E\to A$ 
from an elliptic curve $E$ such that 
$(X,\Delta) \times_A E \cong (F,\Delta|_F) \times_{k} E$, 
where $F$ is any closed fiber of $a$. 

In the above statements the base-changes to $\Delta$, to $E$ or to $F$, are taken as a $\mathbb Z_{(p)}$ almost Cartier divisor as in \cite{Har94}.

\end{thm}
The next example shows that  in Theorem \ref{thm:isotrivial} one cannot have the base-change outside of the case of $\mathbb{F}_p$ in general:
\begin{eg} \label{rem:char0}
Choose $k\neq \overline{\mathbb F}_p$. 
Let $E$ be an elliptic curve 
and let $\mathcal L$ be a non-torsion line bundle of degree zero on $E$.
Put $X:=\mathbb P(\mathcal L \otimes \mathcal O_E)$. Then $-K_X$ is nef. 
However, for any finite morphism $\pi:E'\to E$ from an elliptic curve $E'$, 
$\pi^*\mathcal L$ is also not a torsion line bundle, 
so $X\times_E E' \to E'$ does not split. 
\end{eg} 
\subsection{The Albanese morphism when $-K_X$ ample.}
%
One might wonder if it is possible to say a more precise statement than that of Theorem \ref{thm:main} in the case of $-K_X-\Delta$ ample. When $X$ is smooth, then it is known in this case that $X$ is rationally chain connected (RCC) \cite{KMM92}, so the Albanese variety has dimension $0$. However, the RCC property is famously not known for singular $X$, and hence also the Albanese variety being trivial is not known. We show this here:

\begin{thm} \label{thm:ample}
Let $(X,\Delta) $ be a  projective pair such that 
$K_X+\Delta$ is $\mathbb Z_{(p)}$-Cartier.
Suppose either that 
\begin{enumerate}
\item \label{itm:ample:ample} $-(K_X+\Delta)$ is ample and $(X,\Delta)$ is $F$-pure, or 
\item \label{itm:ample:nef_big} $-(K_X+\Delta)$ is nef and big and $(X,\Delta)$ is strongly $F$-regular. 
\end{enumerate}
Then the Albanese morphism of $X$ is a constant map. 
\end{thm}

Interestingly this algebraic proof is the only one that we are aware of that shows the following topological vanishing, which is a direct consequence of Theorem \ref{thm:ample}, using \cite[Prop 2.7]{PZ19}.

\begin{cor}
In the situation of \autoref{thm:ample}, $H^1_{\textrm{\'et}}(X, \mathbb Q_l)=0$ for $l \neq p$.
\end{cor}

\subsection{Outline of the argument}
Consider the situation of Theorem \ref{thm:main}, and let $Y$ be the normalization of $a(X)$. Denote by $h : X \to Y$ the induced morphism. For simplicity, let us assume that every space appearing in the proof is smooth, and that $\Delta=0$. In particular, then all the canonical divisors showing up in the proof are Cartier.  

The proof of Theorem \ref{thm:main} starts by reducing it to the statement that $\omega_Y^*$ is weakly positive, in Proposition \ref{prop:main}. As in this outline we assume that all varieties  are smooth, and hence so is $Y$, here this is a reduction to the statement that $-K_Y$ pseudo-effective (see Definition~\ref{defn:wp}).

The proof of Proposition \ref{prop:main} contains multiple steps, using multiple earlier results as well as peculiar features of the situation. We prefer to point the reader to the actual proof instead of detailing these different parts here.

Therefore, we arrived to the point that it suffices to prove that $-K_Y$ is pseudo-effective or equivalently that $-h^* K_Y$ is pseudo-effective. To explain the main idea behind this proof, one has to get a little bit technical. First, consider the following commutative diagram for each $e \geq 0$, where 
\begin{itemize}
\item $X^e=X$ and $Y^e= Y$ and the upper index is used only to distinguish the source of the $e$-times iterated Frobenii from their target, 
\item $h^e=h$ via the identifications of the previous point, 
\item $W_e$ is the normalization of the reduced subscheme of $X \times_Y Y^e$, and
\item $\alpha_e$, $\beta_e$ and $\gamma_e$ are the induced morphisms:
\end{itemize}
\begin{equation*}
\xymatrix@C=70pt{
X^e    \ar[rd]^{\gamma_e} \ar@/_1pc/[ddr]_{h^e} \ar@/^1pc/[drr]^{F^e_X} \\
 & W_e \ar[d]_{\alpha_e } \ar[r]_{\beta_e} & X \ar[d]^h \\
& Y^e \ar[r]^{F_Y^e} & Y
}
\end{equation*}
Consider further the ``magic map'' $m_e$ defined in the following diagram
\begin{equation*}
\xymatrix@C=10pt{
\gamma_{e,*} \sO_{X^e} \big(K_{X^e/X} \big) \ar[d]_{\Tr_{\gamma_e} \otimes \beta_e^* \omega_X^{-1}} \ar[rrrrrd]^{m_e} \\
\sO_{W_e} \big( K_{W_e/X} \big) \ar@{^(->}[rrrrr]_(0.43){\sO_{W_e} \big( K_{W_e/X} \big) \otimes \gamma_e^{\#}} & & & & &  \sO_{W_e} \big( K_{W_e/X} \big)\otimes \gamma_{e,*} \sO_X  \ar@{}[r]|(0.53){\Large \cong} & \gamma_{e,*} \sO_{W_e}  \left( \gamma_e^* K_{W_e/X} \right) 
}
\end{equation*}
We claim that $m_e$ is not zero. As $\sO_{W_e} \big( K_{W_e/X} \big) \otimes \gamma_e^{\#}$ is injective, for that it is enough to show that $\Tr_{\gamma_e} \otimes \beta_e^* \omega_X^{-1}$ is non-zero, for which it is enough to show that its pushforward via $\beta_e$ is non-zero. This latter non-zero statement follows from  the next commutative diagram, taking into account that the composition is surjective by the smoothness of $X$ (which we assumed here for simplicity, but the general strongly $F$-regular hypothesis in Theorem \autoref{thm:main} also allows the same argument):
\begin{equation*}
\xymatrix@C=70pt{
F_{X,*}^e \sO_{X^e} \big(K_{X^e/X} \big)
\ar[r]_{ \omega_X^{-1} \otimes \beta_{e,*} \Tr_{\gamma,e}}
\ar@/^1.5pc/[rr]^{\omega_X^{-1} \otimes\Tr_{F^e_X}}
&
\beta_{e,*} \sO_{W_e} \big( K_{W_e/X} \big)
\ar[r]_(0.6){\omega_X^{-1} \otimes\Tr_{\beta_e}}
&
\sO_X
}
\end{equation*}
Having explained that $m_e$ is non-zero, we work towards explaining why 
 $m_e$ is the key to the our proof. The basic reason is that it allows to transfer positivity from the source to the target. To explain why that is the case, we need to compute the canonical divisors on both ends of $m_e$. For the source this is quick. We have:
\begin{equation}
\label{eq:outline:m_e_source}
K_{X^e/X} = K_{X^e} - F^{e,*} K_X = (1-p^e) K_X,
\end{equation}
where in the last equality we used the identification $X =X^e$. 

For the target of $m_e$ this computation is a little longer. First, let $r>0$ be the smallest integer, for which the generic fiber of $\alpha_r$ is geometrically normal. It follow then that $W_e \to W_r \times_{Y^r} Y^e$ is an isomorphism along the generic fiber of $\alpha_e$ for any integer $e \geq r$, and hence for any such $e$ we have
\begin{equation}
\label{eq:outline:m_e_target}
\gamma_e^* K_{W_e/X} = F_X^{e-r,*} \gamma_r^* K_{W_r/X} + \gamma_e^* K_{W_e/W_r} \leq F_X^{e-r,*} \gamma_r^* K_{W_r/X} + h^* (1-p^{e-r}) K_Y,
\end{equation}
where the last inequality is worked out in the following line:
\begin{equation*}
K_{W_e/W_r} \leq K_{W_r \times_{Y^r} Y^e /W_r} = \alpha_e^* K_{Y^e/Y^r} = \alpha_e^* (1-p^{e-r}) K_Y
\end{equation*}
Putting together \eqref{eq:outline:m_e_source} and \eqref{eq:outline:m_e_target}, we obtain that $m_e$ is a non-zero $\sO_{W_e}$-linear homomorphism
\begin{equation}
\label{eq:outline:m_e_final}
m_e : \gamma_{e,*} \sO_X((1-p^e) K_X) \to F_X^{e-r,*} \gamma_r^* K_{W_r/X} + h^* (1-p^{e-r}) K_Y
\end{equation}
The basic idea of the proof of Theorem \ref{thm:main} is that some variant of $H^0(W_e, m_e)$ should allow to transfer positivity of $-K_X$, which is assumed in Theorem \ref{thm:main}, to positivity of $-h^* K_Y$. The mathematical reason behind this idea is clear on an intuitive level by \eqref{eq:outline:m_e_final}, especially if one disregards the term $F_X^{e-r,*} \gamma_r^* K_{W_r/X}$. However, turning it into a precise proof is rather technical, as one has to solve the following issues (see the proof of Theorem \ref{thm:weak pos} for details): 
\begin{enumerate}
\item \ul{Is $H^0(W_e,m_e)$ non-zero?} The best we can prove is that $H^0\big(m_e \otimes \sO_{W_e} ( \beta_e^* B)\big)$ is non-zero, where $B$ is a globally generated ample divisor on $X$. (See Step \ref{step:v} of the proof of Theorem \ref{thm:weak pos}; to match the notations there, one has to replace $X$ by $U$ and $Y$ by $V$). In fact, by looking at the proof, we can even allow a twist by an arbitrary nef divisor $N$ on $W_e$. That is, for every such $N$, $H^0\big(m_e \otimes \sO_{W_e} ( N + \beta_e^* B)\big)$ is non-zero. Using a few projection formulas and the expression of $m_e$ in \eqref{eq:outline:m_e_final}, we obtain a non-zero map
\begin{multline}
\label{eq:outline:m_e_B_N}
\qquad m_{e,B,N}: \quad
H^0(X, (1-p^e) K_X + p^e B + \gamma_e^* N) \\ \to H^0(X, h^* (1-p^{e-r}) K_Y + F_X^{e-r,*} \gamma_r^* K_{W_r/X} + p^e B + \gamma_e^* N ) 
\end{multline}
\item First let us disregard the term $F_X^{e-r,*} \gamma_r^* K_{W_r/X}$ from \eqref{eq:outline:m_e_B_N}. We will explain later how to deal with this term.  Then, $m_{e,B,N}$ becomes the non-zero $k$-linear map 
\begin{equation}
\label{eq:outline:m_e_B_N_simplified}
\qquad H^0(X, (1-p^e) K_X + p^e B + \gamma_e^* N) \to H^0(X, h^* (1-p^{e-r}) K_Y + p^e B + \gamma_e^* N).
\end{equation}
The issue is that \ul{it is not clear how to obtain Theorem} \ref{thm:main_smooth} \ul{directly  even using this simplified version of $m_{e, B,N}$}.  The main reason is that, asymptotically in $e$, the coefficient of $-K_X$ and of $B$ in the source of \eqref{eq:outline:m_e_B_N_simplified} are the same. However,  to show our theorem we would need to be able to induce sections of $B-h^* a K_Y$ from sections of $B-a K_X $ for any integer $a>0$. 

\item \label{itm:outline:limiting} Our solution to  the problem explained in the previous point is to do the construction of sections step by step. In this point, everything we explain is asymptotic in $e$, that is after taking limit $e \to \infty$. Using, this simplification, assuming that we know that $0 \neq H^0(X,B-h^* a K_Y)$ , we use $m_e$ to show that $0 \neq H^0\bigg(X,B-h^* \left(a+ \frac{1}{p^r} \right) K_Y\bigg)$. Iterating this, shows that $B - ah^* K_Y$ is effective for $a \to \infty$,  and hence $-h^* K_Y$ is pseudo-effective. 

The main issue is that during this argument, \ul{we end up needing to twist $m_e$ by sheaves of the form $- a' h^* K_Y$, which are not necessarily nef}. Hence, we are not able to show directly that the twist of $m_e$ induces a non-zero map on sections. Instead, we show that pushforwards of adequate twisted versions of $m_e$ are non-zero. Then, we show that the sources of these maps are weakly-positive sheaves, with controlled generic global generation properties. In particular, their images via non-zero maps to torsion-free sheaves also have non-zero sections. The key to these arguments are the recent semi-positivity statements showed by the first author in \cite{Eji19d}. The version used in the present article is stated  in Proposition \ref{prop:genglgen}, which is made more precise in Proposition \ref{prop:B+E}. Besides the discovery of the ``magic map'' $m_e$,  this is the main recent technical development that made the argument of the article possible. 

We note that this argument requires a very careful adjustment of the setting and also of the general setup. Hence, the proof of Theorem \ref{thm:weak pos} is admittedly technical. It is hard to explain all these setup fine-adjustments deeper in the framework of an outline. Therefore, we refer the reader to the actual proof of Theorem \ref{thm:weak pos} for these precise fine-tunings, especially to Step \ref{step:F} of that proof. 

\item We are left to explain briefly \ul{why one could discard the term $F_X^{e-r,*} \gamma_r^* K_{W_r/X}$.} The reason is that one can show that this term is anti-effective. The proof is based on writing down the canonical bundle formula for $K_{W_i/W_{i-1}}$ using foliations, see Section \ref{sec:foliations}, as well as Step \ref{step:negative} of the proof of Theorem \ref{thm:weak pos}. The main point is that minus the determinant of the saturated image of $\alpha_i^* \Omega_{Y} \to \Omega_{W_i}$ comes into the picture, which is anti-effective as $\sO_A^{\oplus \dim A}|_{Y} \cong \Omega_A|_Y \to \Omega_Y$ is generically surjective. 
\end{enumerate}

The most subtle part of the proof is the  limiting argument explained in point \eqref{itm:outline:limiting}. According to the best knowledge of the second author, the first birational geometry argument containing a limiting argument with a repeated use of Kawamata-Viehweg vanishing  is \cite[Thm~6.3]{HM10}, and the first article using limiting argument with a repeated use of semi-positivity statements is \cite[Thm 3.10]{Pat14}. 

We note that all the above considerations are put together in Step \ref{step:T} of the proof of Theorem \ref{thm:weak pos}. This yields Theorem \ref{thm:weak pos}, from which all points except the flatness statement of Theorem \ref{thm:main} follow in a few lines. See the proof of Theorem \ref{thm:main}. The  flatness statement is shown in the same proof, and it is lifted with minor modifications from \cite[Section 4]{PZ19}, which in turn is lifted with minor modifications from \cite[Theorem]{LTZZ}. The main difference is that the semi-positivity input therein is replaced by the statement of Theorem~\ref{thm:weak pos}.  As the argument is quite short, we refer the reader to the actual proof. 

The proof of Theorem \ref{thm:isotrivial} is much simpler than that of Theorem \autoref{thm:main}. It follows very closely the proof of the main statements of \cite{PZ19}, modifying that where it is necessary. The two main points of modifications are the following:
\begin{itemize}
\item The semi-positivity engines used in \cite{PZ19} are replaced by the newer ones that do not assume good singularities of general fibers.
\item The Frobenius pull-backs are avoided by the classification of vector bundles on elliptic curves \cite{Oda71}, which is necessary as Frobenius pull-backs destroy singularities if the general fibers are highly singular. 
\end{itemize}

\begin{ack}
The first author was partly supported by MEXT Promotion of Distinctive Joint Research Center Program JPMXP0619217849. The second author is supported by the grant 200020B/192035 from the Swiss National Science Foundation the ERC starting grant 804334.
\end{ack}

\section{Notations and definitions} 
Throughout this paper, we work over an algebraically closed field $k$ of characteristic $p>0$. 
By a \textit{$k$-scheme} we mean a separated scheme of finite type over $k$. 
A \textit{variety} is an integral $k$-scheme. 
A \textit{curve} (resp. \textit{surface}) is a variety of dimension one 
(resp. two). 

Let $\phi:S\to T$ and $T'\to T$ be morphisms of $k$-schemes. 
We denote by $S_{T'}$ the fiber product $S\times_T T'$ of $S$ and $T'$ over $T$. 
The induced morphism $S_{T'}\to T'$ by $\phi$ is denoted by $\phi_{T'}$. 

Let $X$ be a $k$-scheme. Let $e\ge1$ be a number. 
We denote by $F^e_X:X^e\to X$ the absolute Frobenius morphism of $X$ iterated $e$-times. 
Note that the source of $F_X^e$ is denoted by $X^e$ 
to distinguish from the target of $F_X^e$. 

Let $f:X\to Y$ be a morphism of $k$-schemes. 
We denote $f:X\to Y$ by $f^{e}:X^e\to Y^e$ when we regard $X$ and $Y$ as $X^e$ and $Y^e$. 
We define the morphism $F_{X/Y}^{e}:X^e\to X_{Y^e}=X\times_Y Y^e$ 
by $F_{X/Y}^{e}:=F_X^e\times_Y f^{e}$:
$$
\xymatrix{ 
X^e \ar@/_50pt/[dd]_{f^{e}} \ar[dr]^{F_X^e} \ar[d]_{F_{X/Y}^{e}} & \\
X_{Y^e} \ar[r] \ar[d]_{f_{Y^e}} & X \ar[d]^f \\
Y^e \ar[r]^{F_Y^e} & Y. 
}
$$

Let $X$ be a normal variety. A \textit{$\mathbb Q$-Weil divisor} (resp. \textit{$\mathbb Z_{(p)}$-Weil divisor}) on $X$ is an element of $\mathrm{Weil}(X) \otimes_{\mathbb Z} \mathbb Q$ (resp. $\mathrm{Weil}(X) \otimes_{\mathbb Z}\mathbb Z_{(p)}$), where $\mathrm{Weil}(X)$ is the additive group of Weil divisors on $X$. We say that a $\mathbb Q$-Weil divisor $\Delta$ is \textit{$\mathbb Q$-Cartier} (resp. \textit{$\mathbb Z_{(p)}$}-Cartier) if $i\Delta$ is a Cartier divisor for an integer $i>0$ (resp. an integer $i>0$ not divisible by $p$). 

A \emph{pair} $(X, \Delta)$  is a normal variety $X$ together with an effective $\bQ$-Cartier Weil divisor $\Delta$ on $X$. A pair $(X, \Delta)$ is affine (resp. projective) if $X$ is affine (resp. projective)

\medskip

Next, we collect several notions used in the proof of the main theorem. 
\begin{defn} \label{defn:F-sing}
Let $(X, \Delta)$ be an \textit{affine} pair. 
\begin{enumerate}
\item \cite{HR76, HW02} 
We say that  $(X,\Delta)$ is \textit{$F$-pure} if 
the composition of 
$$
\mathcal O_X 
\xrightarrow{{F_X^e}^\sharp} {F_X^e}_*\mathcal O_X
\hookrightarrow {F_X^e}_*\mathcal O_X\left(\lfloor (p^e-1)\Delta\rfloor\right)
$$
splits as an $\mathcal O_X$-module homomorphism for every $e\in\mathbb Z_{>0}$. 
\item \cite{HH89, HW02}
We say that $(X,\Delta)$ is \textit{strongly $F$-regular} if 
for every effective Weil divisor $D$ on $X$, 
there exists an $e\in\mathbb Z_{>0}$ such that the composition of 
$$
\mathcal O_X 
\xrightarrow{{F_X^e}^\sharp} {F_X^e}_*\mathcal O_X
\hookrightarrow {F_X^e}_*\mathcal O_X\left(\lceil (p^e-1)\Delta +D\rceil\right)
$$
splits as an $\mathcal O_X$-module homomorphism. 
\end{enumerate}
 We say that a pair $(X,\Delta)$ is \textit{$F$-pure} (resp. \textit{strongly $F$-regular}) if there exists an affine open cover $\{U_i\}$ of $X$ such that every $\left(U_i,\Delta|_{U_i}\right)$ is $F$-pure (resp. strongly $F$-regular). 
\end{defn}
\begin{defn} \label{defn:wp}
Let $Y$ be a normal quasi-projective variety and let $\mathcal F$ and $\mathcal G$ be coherent sheaves on $Y$. 
\begin{enumerate}
\item We say that a morphism $\varphi:\mathcal F\to \mathcal G$ is \textit{generically surjective} (resp. \textit{generically isomorphic}) if the induced morphism $\varphi_\eta:\mathcal F_\eta \to \mathcal G_\eta$ between stalks at the generic point $\eta$ of $Y$ is surjective (resp. an isomorphism). 
\item We say that $\mathcal F$ is \textit{generically generated by its global sections} if the natural morphism 
$$
H^0(Y, \mathcal F) \otimes_{\mathrm{Spec}\,k} \mathcal O_Y 
\to \mathcal F
$$
is generically surjective. 
\item (\textup{\cite[Definition~3.1]{Vie83II}})
We say that $\mathcal G$ is \textit{weakly positive} if for every ample line bundle $\mathcal H$ and each number $\alpha \ge 1$, there exists a number $\beta \ge 1$ such that 
$$
\left( \mathrm{Sym}^{\alpha\beta}(\mathcal F) \right)^{**}
\otimes \mathcal H^\beta
$$
is generically generated by its global sections. 
Here, $\mathcal G^{**}$ denotes the double dual of $\mathcal G$. 
\end{enumerate}
\end{defn}
\section{Varieties of maximal Albanese dimension with weakly positive anti-canonical}
In the proof, Proposition~\ref{prop:main} plays an important role. 
To prove Proposition~\ref{prop:main}, we need the following proposition. 
\begin{prop}[\textup{\cite[Proposition~1.4]{HPZ19}}] \label{prop:HPZ}
Let $Y$ be a normal projective variety such that 
$\dim H^0\left(Y, \omega_Y^{[n]}\right) \le 1$ for each $n\ge 1$. 
If the Albanese morphism $a:Y \to A$ is generically finite, 
then $a$ is separable. 
\end{prop}
\begin{prop} \label{prop:main}
Let $Y$ be a normal projective variety of maximal Albanese dimension such that 
$\omega_Y^*$ is weakly positive. 
Then $Y$ is an abelian variety. 
\end{prop}
\begin{proof} We divide the proof into six steps.
\begin{step} \label{step:kappa}
We prove that $\dim H^0\left(Y,\omega_Y^{[n]}\right) \le 1$ for each $n\ge 1$. 
Assume that there is an effective Weil divisor $C \sim_{\mathbb Z} nK_Y$ 
for some $n\ge 1$. 
Since $\mathcal O_Y(-K_Y)$ is weakly positive by the assumption, 
we see that $\mathcal O_Y(-C)$ is also weakly positive, 
so \cite[Lemma~4.4]{Eji19w} shows that $C=0$. 
This means that $\dim H^0\left(Y,\omega_Y^{[n]}\right) \le1$ for each $n\ge 1$. 
\end{step}
\begin{step} \label{step:K_Y=0}
We show that $K_Y=0$. 
Let $g:Y\to A$ denote the Albanese morphism of $Y$. 
Let $Z$ be the image of $g$. 
Let $\gamma: Y\to Z$ be the induced morphism. 
By Step~\ref{step:kappa} and Proposition~\ref{prop:HPZ}, 
we see that $\gamma$ is separable, 
so there is a generically surjective morphism $g^*\Omega_A^1 \to \Omega_Y^1$. 
Since $\Omega_A^1$ is a trivial vector bundle, 
we conclude that $\Omega_Y^1$ is generically generated by its global sections, 
which means that $K_Y \ge 0$. 
Since $\mathcal O_Y(-K_Y)$ is weakly positive, 
\cite[Lemma~4.4]{Eji19w} tells us that $K_Y=0$. 
\end{step}
\begin{step} \label{step:det}
Put 
$
\mathcal D:=\left( \det\left( {F_Y}_*\mathcal O_Y \right) \right)^{**}. 
$
We prove that $\mathcal D^* \cong \mathcal D$.
We may assume that $Y$ is smooth. 
Since $\omega_Y \cong \mathcal O_Y$ by Step~\ref{step:K_Y=0}, 
we obtain from the Grothendieck duality that 
\begin{align*}
\mathcal Hom \left( {F_Y}_*\mathcal O_Y, \mathcal O_Y \right)
\cong {F_Y}_* \mathcal Hom \left(\mathcal O_Y, \mathcal O_Y \right)
\cong {F_Y}_*\mathcal O_Y, 
\end{align*}
so 
$
\mathcal D^* 
\cong \det\left(\left({F_Y}_*\mathcal O_Y\right)^*\right)
\cong \det\left({F_Y}_*\mathcal O_Y\right)
=\mathcal D.
$
\end{step}
\begin{step} \label{step:smooth}
We show that $Y$ is smooth. 
We use the notation in the following diagram: 
$$
\xymatrix{
Y^1 \ar[d]_{F_{Y/Z}} \ar[dr]^{F_Y} & \\ 
Y_{Z^1} \ar[d]_{\gamma_{Z^1}} \ar[r]_{w} & Y \ar[d]^\gamma \\
Z^1 \ar[r]^{F_Z} & Z
}
$$
Then, we have the following morphisms:
$$
\gamma^* \left( {F_Z}_*\mathcal O_Z \right)
\to w_* \mathcal O_{Y_{Z^1}}
\xrightarrow{w_* \left({F_{Y/Z}}^\sharp\right)} 
w_* {F_{Y/Z}}_* \mathcal O_{Y} = {F_Y}_*\mathcal O_Y. 
$$
The composition is generically surjective, 
since $\gamma$ is a generically finite separable surjective morphism by Proposition~\ref{prop:HPZ}. 
Combining this with 
$
g^* \delta: g^* \left( {F_A}_*\mathcal O_A \right)
\twoheadrightarrow 
\gamma^* \left( {F_Z}_*\mathcal O_Z \right), 
$
where $\delta:{F_A}_*\mathcal O_A \twoheadrightarrow {F_A}_* \mathcal O_Z \cong {F_Z}_*\mathcal O_Z$
is obtained as the direct image of the natural morphism $\mathcal O_A \twoheadrightarrow \mathcal O_Z$ by $F_A$, 
we get the generically surjective morphism 
$$
\epsilon: g^* \left( {F_A}_*\mathcal O_A \right) 
\to {F_Y}_* \mathcal O_Y. 
$$
Thanks to \cite[Corollary~1.7]{Oda71}, we see that ${F_A}_*\mathcal O_A$ is a homogeneous vector bundle, so there is a filtration 
$$
0 = \mathcal E'_0 
\subset \mathcal E'_1 
\subset \cdots 
\subset \mathcal E'_{p^{\dim A}-1} 
\subset \mathcal E'_{p^{\dim A}} 
= g^* \left( {F_A}_* \mathcal O_A \right) 
$$
whose each quotient $\mathcal E'_{i+1}/\mathcal E'_i$ is 
an algebraically trivial line bundle. 
For each $0 \le j < p^{\dim Y}$, 
let $i_j$ be the largest number such that 
$\epsilon\left(\mathcal E'_{i_j}\right)$ is of rank $j$. 
Let $\mathcal E_j$ be the saturation in ${F_Y}_*\mathcal O_Y$ 
of $\epsilon(\mathcal E'_{i_j})$. 
Then we obtain the filtration 
$$
0 = \mathcal E_0 
\subset \mathcal E_1 
\subset \cdots 
\subset \mathcal E_{p^{\dim Y}-1} 
\subset \mathcal E_{p^{\dim Y}} 
= {F_Y}_* \mathcal O_Y
$$
such that for each $0\le j < p^{\dim Y}$, 
\begin{itemize}
\item $\mathcal E_{j+1}/\mathcal E_j$ is a torsion-free coherent sheaf of rank one and 
\item there is an algebraically trivial line bundle $\mathcal L_j$ and an injective morphism $\iota_j:\mathcal L_j \hookrightarrow \mathcal E_{j+1}/\mathcal E_j$ for each $0\le j <p^{\dim Y}$. 
\end{itemize}
Note that $\mathcal E_0=0$, since ${F_Y}_*\mathcal O_Y$ is torsion-free. 
Let $V \subseteq Y$ be a smooth open subset such that 
$\mathrm{codim}_Y(Y\setminus V) \ge 2$ 
and each $\left(\mathcal E_{j+1}/\mathcal E_j\right)|_V$ is a line bundle. 
Let $C \subseteq V$ be a general projective curve. 
Then 
$$
\bigotimes_{j=0}^{p^{\dim Y}-1} \mathcal L_j|_C
\overset{\bigotimes \iota_j|_C}{\hookrightarrow}
\bigotimes_{j=0}^{p^{\dim Y}-1} \left( \mathcal E_{j+1}/\mathcal E_j \right)|_C
\cong 
\det \left( {F_V}_*\mathcal O_V\right) |_C.  
$$
Since each $\mathcal L_j|_C$ and 
$\det \left( {F_V}_*\mathcal O_V\right)|_C =\mathcal D|_C$ 
are line bundles of degree zero by Step~\ref{step:det}, 
we see that each $\iota_j|_C$ is an isomorphism, which implies that 
$
\mathcal L_j|_V \overset{\iota_j|_V}{\hookrightarrow} 
(\mathcal E_{j+1}/\mathcal E_j)|_V
$ 
is also an isomorphism. 
Thus the composition of
$$
\mathcal L_j \overset{\iota_j}{\hookrightarrow} \mathcal E_{j+1}/\mathcal E_j 
\hookrightarrow (\mathcal E_{j+1}/\mathcal E_j)^{**}
$$
is an isomorphism, and hence so is $\iota_j$. 
In particular, $\mathcal E_{j+1}/\mathcal E_j$ is a line bundle. 
This means that ${F_Y}_*\mathcal O_Y$ is locally free, 
and thus $Y$ is smooth by Kunz's theorem. 
\end{step}
\begin{step} \label{step:trivial}
We prove that $\Omega_Y^1$ is a trivial vector bundle. 
As we have seen in Step~\ref{step:K_Y=0}, 
$\Omega_Y^1$ is generically generated by its global sections, 
so there is a generically isomorphic morphism 
$\lambda:\bigoplus^{\dim Y}\mathcal O_Y \to \Omega_Y^1$. 
Taking the determinant, we obtain the generically isomorphic morphism 
$\det(\lambda): \mathcal O_Y \to \omega_Y$. 
Then by Step~\ref{step:K_Y=0}, we see that $\det(\lambda)$ is an isomorphism, 
which means that so is $\lambda$. 
\end{step}
\begin{step} \label{step:abelian}
We show the assertion. 
By \cite[Lemma~(1.3)]{MS87}, the morphism 
$$
g^*:H^0(A,\Omega_A^1) \to H^0(Y,\Omega_Y^1)
$$ 
is injective. 
This is an isomorphism, since 
$$
\dim A = \dim H^0(A,\Omega_A^1) 
\le \dim H^0(Y,\Omega_Y^1) = \dim Y 
$$
and $\dim Y\le \dim A$ by the assumption. 
Taking into account Step~\ref{step:trivial}, we get that $g^*\Omega_A^1 \to \Omega_Y^1$ is an isomorphism, 
which means that $g$ is \'etale, 
so $Y$ is an abelian variety by \cite[Section~18, Theorem]{Mum70}. 
\end{step}
\end{proof}
\begin{rem}
Steps~\ref{step:trivial} and~\ref{step:abelian} in the above proof can be skipped by using \cite[Theorem~0.3]{HPZ19}, which shows that if $X$ is a smooth projective variety with $\kappa(X)=0$ and if the Albanese morphism $a:X\to A$ is generically finite, then $a$ is birational. 
\end{rem}
\section{Foliations associated to  fibrations}
\label{sec:foliations}
Below we define some foliations associated to surjective fibrations. These foliations  will be used in the proof of Theorem \ref{thm:weak pos}, and are primarily needed as there $f$ is not assumed to be separable, in one of the cases. In fact, in the separable $f$ case, the proof of Theorem \ref{thm:weak pos} can be done without the use of Lemma \ref{lem:foliation}. 

\begin{notation}
\label{notation:foliation}
Let $g:X\to Z$ be a surjective morphism between normal varieties. 
Let $W_e$ be the normalization of the reduced subscheme of $X \times_Z Z^e$, and introduce the following names for the induced morphisms:
\begin{equation}
\label{eq:diagram}
\xymatrix@C=70pt{
X^e  \ar[rd]_{\gamma_{e-1}^1}  \ar[rrd]^{\gamma_e} \ar@/_3pc/[ddrr]_{g^e} \ar@/^2pc/[drrr]^{F^e_X} \\
& W_{e-1}^1 \ar[r]_{\delta_e} \ar[dr]_{\alpha_{e-1}^1} & W_e \ar[d]_{\alpha_e } \ar[r]_{\beta_e} & X \ar[d]^g \\
& & Z^e \ar[r]^{F_Z^e} & Z
}
\end{equation}
Note that $W_e=W_{e-1}^1/\sF^e$ for the $p$-foliation 
\begin{equation*}
\sT_{W_{e-1}^1} \supseteq \sF^e  = \big( \coker (d \delta_e) \big)^*,
\end{equation*}
where $d \delta_e : \delta_e^* \Omega_{W_e}\to \Omega_{W_{e-1}^1}$ is the induced homomorphism on differentials by $\delta_e$. Set then
\begin{equation*}
\sG^e \overset{\mathrm{def}}{=} \big( \coker (d \alpha^1_{e-1}) \big)^*.
\end{equation*}  
\end{notation}

\begin{lem} \label{lem:foliation}
Using Notation \ref{notation:foliation}, we have $\sF^e = \sG^e$, and hence $W_e = W_{e-1}^1/\sG^e$.
\end{lem}

\begin{proof}
By \eqref{eq:diagram}, we have $ \alpha_{e-1}^1 = \alpha_e \circ \delta_e $, and hence $\sF^e  \subseteq \sG^e$. If $W^1_{e-1}/\sG^e$ has induced morphisms to $Z^e$ and to $W_{e-1}$ that fits into the following commutative diagram, then by the minimality of the normalization of the reduced subscheme of the fiber product we have to have $\sG^e= \sF^e$:
\begin{equation*}
\xymatrix{
W_{e-1}^1 \ar[rrd]  \ar[rd] \ar[rdd] \\
& W^1_{e-1}/\sG^e \ar[r] \ar[d] & W_{e-1} \ar[d]^{\alpha_{e-1}} \\
& Z^e \ar[r] & Z^{e-1} 
}
\end{equation*}
In fact, for this it is enough to show that $\sG^e$ kills all the local functions on any affine open set of both $W_{e-1}$ and for $Z^e$:
\begin{itemize}
\item  for $W_{e-1}$ this holds just by the virtue of $\sG^e$ being a $p$-foliation, and
\item  for $Z^e$, let $\Spec A \subseteq Z^e$ and $\Spec B \subseteq W_{e-1}^1$ be two affine opens such that $\alpha_{e-1}^1 (\Spec B) \subseteq \Spec A$. Fix also $D \in\Gamma(\Spec B, \sG^e)$, and view it either as a derivation on $B$ or as a functional on $\Omega_{B/k}$ depending on whether we evaluate on ring elements or on differentials. The connection between the two points of view is given by $D(b)=D(db)$, for every $b \in B$. We need to show that $D(a)=0$ for every $a \in A$. However, by the definition of $\sG^e$, we have $D(da)=0$, and then by the connection between the two points of view we have $D(a)=D(da)=0$. 
\end{itemize}

\end{proof}
\section{The main technical theorem}
Next, we prove Proposition~\ref{prop:B+E}, which is the key in the proof of Theorem~\ref{thm:main}. Proposition~\ref{prop:B+E} is an application of the following positivity result:
\begin{prop}[\textup{\cite[Corollary~6.5, Lemma~5.4 and Example~5.8]{Eji19d}}] \label{prop:genglgen}
Let $(U, \Delta)$ be a pair with a surjective projective morphism $h:U\to V$ be a to a normal quasi-projective variety. Let $U_{\eta}$ denote the generic fiber of $h$. 
Suppose that 
\begin{itemize}
\item $\left(U,\Delta|_{U_\eta}\right)$ is $F$-pure, 
\item $i(K_U+\Delta)|_{U_\eta}$ is an ample Cartier divisor for some integer $i>0$ not divisible by $p$, and 
\item there exists a generically finite projective surjective morphism from $V$ to a closed subset of a dense open subset of an abelian variety $A$. 
\end{itemize}
Set $\mathcal F_m := h_*\mathcal O_U(m(K_U+\Delta))$ 
for each $m\ge 1$ with $i|m$. Fix also a symmetric ample divisor $H'$ $($i.e. $(-1)_A^* H' =H')$ on $A$ 
with $|H'|$ free and let $H$ be the pullback of $(\dim V +1)H'$ to $V$.

Then, there exists an integer $m_0\ge 1$ such that the following hold:
\begin{enumerate}
\item \label{itm:genglgen:diminished}
The set $\mathbb B_-(\mathcal F_m)$ is contained in a proper closed set of $V$ 
for each $m\ge m_0$ with $i|m$ 
$($for the definition of $\mathbb B_-$, see \cite[\S 4]{Eji19d}$)$. 
\item  \label{itm:genglgen:generated}
The coherent sheaf $\mathcal F_m(H)$ is generically generated by its global sections 
for each $m\ge m_0$ with $i|m$. 
\end{enumerate}
\end{prop}
\begin{prop} \label{prop:B+E}
Consider the following varieties and morphisms:
\begin{itemize} 
\setlength{\leftskip}{-10pt}
\item Let $(X, \Delta)$ be a normal projective pair with  $\Delta$ being a $\mathbb Z_{(p)}$-Weil divisor.
\item Let $f:X\to Y$ be a surjective morphism to a normal projective variety $Y$ of maximal Albanese dimension. 
\item Let $V$ be a dense open subset of $Y$, and 
set $U:=f^{-1}(V)$ and $h:=f|_U:U\to V$. 
\end{itemize}
Additionally, consider the  following divisors with the following properties on the above varieties:
\begin{itemize}
\item Let $B$ be an ample $\mathbb Z_{(p)}$-Cartier divisor on $X$. 
\item Let $D$ be a $\mathbb Z_{(p)}$-Weil divisor on $X$ such that 
	$D-(K_X+\Delta)$ is a nef $\mathbb Z_{(p)}$-Cartier divisor. 
\item Let $E \ge 0$ be a $\mathbb Z_{(p)}$-Weil divisor on $U$.
\item Suppose that 
\begin{itemize}
\item[$\circ$] $\left(X_\eta,\Delta|_{X_\eta}\right)$ is 
strongly $F$-regular, where $X_\eta$ is the generic fiber of $f$, and 
\item[$\circ$] $D|_{X_\eta}$ and $E|_{X_\eta}$ are 
	nef $\mathbb Z_{(p)}$-Cartier divisors. 
\end{itemize}
\item Take a Cartier divisor $H$ on $V$ as in Proposition~\ref{prop:genglgen}. 
\end{itemize}
Then, there exist positive integers $c$ and $d$ such that 
\begin{itemize}
\item $p\nmid c$ and $p\nmid d$, 
\item $cB$, $cD$ and $cE$ are integral, and 
\item there exists an integer $m_0\ge 1$ such that 
$$
\big( h_* \mathcal O_U(cm (B|_U+dD|_U+E)) \big) \otimes \mathcal O_V(H)
$$ 
is generically generated by its global sections for each $m \ge m_0$. 
\end{itemize}
\end{prop}
\begin{proof}
Since $\left(X_\eta, \Delta|_{X_\eta}\right)$ is strongly $F$-regular, 
there is a $d\in\mathbb Z_{>0}$ with $p\nmid d$ such that 
$$
\left( X_\eta, \left. \left(\Delta+\frac{1}{d}E\right)\right|_{X_\eta} \right)
$$ 
is $F$-pure. 
As $B$ is ample and $D-(K_X+\Delta)$ is nef, 
there is a $c\in\mathbb Z_{>0}$ such that 
\begin{itemize}
\item $p\nmid c$, 
\item $c(K_X+\Delta)$, $cB$, $cD$ and $cE$ are integral, and  
\item $c(B +dD -d(K_X +\Delta))$ is a very ample Cartier divisor. 
\end{itemize}
By \cite[Corollary~6.10]{SW13}, we can find an effective Cartier divisor 
$$
G \sim_{\mathbb Z} c(B +dD -d(K_X +\Delta))
$$ 
such that 
$$
\left(U, \Delta':=\Delta|_U +\frac{1}{d}E +\frac{1}{cd}G|_U \right)
$$ 
is $F$-pure along $X_\eta$. Then we have 
\begin{align*}
cd(K_U +\Delta') 
& = cd(K_U +\Delta|_U ) +cE +G|_U 
\\ & \sim_{\mathbb Z} 
cd(K_U +\Delta|_U ) +cE + cB|_U +cdD|_U -cd(K_U +\Delta|_U)
\\ & =c(B|_U +dD|_U +E), 
\end{align*}
and so 
$$
h_* \mathcal O_U(cdm(K_U +\Delta'))
\cong 
h_* \mathcal O_U (cm(B|_U +dD|_U +E))
$$
for each $m > 0$. 
Since $B|_{X_\eta}$ is ample and $D|_{X_\eta}$ and $E|_{X_\eta}$ are nef, 
$(K_U +\Delta')|_{X_\eta}$ is ample. 
Hence, the assertion follows from Proposition~\ref{prop:genglgen}.\eqref{itm:genglgen:generated}. 
\end{proof}
\begin{thm} \label{thm:weak pos}
Consider the following varieties and morphisms:
\begin{itemize}
\item Let $(X, \Delta)$ be a  projective pair with $\Delta$ a $\mathbb Z_{(p)}$-divisor, 
\item Let $f \colon X \to Y$ be a surjective morphism to  a normal projective variety of maximal Albanese dimension, such that the generic fiber $\left(X_{\eta}, \Delta|_{X_\eta}\right)$ is strongly $F$-regular. 
\item Let $V$ be an open subset of $Y$ such that 
\begin{itemize}
\item[$\circ$] $Y\setminus V$ is of codimension at least two, 
\item[$\circ$] $V$ is regular, and  
\item[$\circ$] $h:=f|_U:U\to V$ is flat, where $U:=f^{-1}(V)$. 
\end{itemize}
\item Assume either that 
\begin{enumerate}
\item \label{itm:weak_pos:separable_fiber} $f$ is separable, or 
\item \label{itm:weak_pos:separable_over_image} the Albanese morphism $a_Y:Y \to A_Y$ of $Y$ is separable over its image. 
\end{enumerate}
\end{itemize}
Additionally, consider the  following divisors and properties of divisors on the above varieties:
\begin{itemize}
\item Let $D$ be a $\mathbb Z_{(p)}$-Weil divisor on $X$ such that 
\begin{itemize}
\item[$\circ$] $D-(K_X+\Delta)$ is a nef $\mathbb Z_{(p)}$-Cartier divisor, 
\item[$\circ$] $D\ge 0$, and 
\item[$\circ$] $D|_{X_\eta} =0$. 
\end{itemize}
\item Put $U_0 := U\setminus \mathrm{Supp}(D)$. 
\end{itemize}
Then $\mathcal O_{U_0}(-h^*K_V|_{U_0})$ is weakly positive. 
\end{thm}
We have the following commutative diagram: 
$$
\xymatrix{
	U_0 \ar@{^(->}[r] & U \ar@{^(->}[r] \ar[d]_-h & X \ar[d]^-f \\
			  & V \ar@{^(->}[r] & Y
}
$$
\begin{proof}
The proof is divided into multiple steps. Each setup starts either with a setup, where we define the necessary notation, or with the precise statement of the goal of the step. In the former case, we also state the precise goal, right after the setup. 
\begin{step} \label{step:F} 
\ul{Setup:} Let $H$ be a Cartier divisor on $V$ 
defined as in Proposition~\ref{prop:genglgen}. 
Take an ample Cartier divisor $B$ on $X$ such that 
\begin{itemize}
\item $B=(\dim X +1)B'$ for some very ample Cartier divisor $B'$ on $X$, 
\item $f_* \mathcal O_X(B)$ is globally generated, 
\item $h^* H \le B|_U$, and 
\item the natural morphism 
$$
f_* \mathcal O_X (B) \otimes f_* \mathcal O_X(B+N) 
\to f_* \mathcal O_X(2B +N)
$$ 
is generically surjective for every Cartier divisor $N$ on $X$ such that 
$N|_{X_\eta}$ is nef. 
\end{itemize}
We can find such a $B$ by using Keeler's relative Fujita vanishing theorem \cite[Theorem~1.5]{Kee03} and the relative Castelnuovo--Mumford regularity \cite[Example~1.8.24]{Laz04I}. 
Note that if $B$ holds the above conditions, 
then so does $lB$ for each $l\in\mathbb Z_{>0}$. 
Set 
\begin{align*}
T := \left\{ t \in \mathbb Z_{(p)} \,\middle| \,
\begin{tabular}{c}
\textup{there is a $\lambda_t\in \mathbb Z_{>0}$ and a $\mathbb Z_{(p)}$-Weil divisor $E_t\ge 0$ on $U$} \\
\textup{such that $B|_U +\lambda_tD|_U -t h^*K_V \sim_{\mathbb Z_{(p)}} E_t$}
\end{tabular}
\right\}. 
\end{align*}
Then $0\in T \ne \emptyset$. 
Fix a $t\in T$. 
Then we have $\lambda_t\in\mathbb Z_{>0}$ and 
$$
0 \le E_t \sim_{\mathbb Z_{(p)}} B|_U +\lambda_tD|_U -th^*K_V.
$$ 
Take a $\delta\in \mathbb Z_{(p)} \cap (0, 1)$. 
Set $B_\delta:= \delta B +D -(K_X +\Delta)$. 
Then $B_\delta$ is an ample $\mathbb Z_{(p)}$-Cartier divisor. 
Let $e_{\delta,t}>0$ be an integer such that 
\begin{itemize}
\item 
$
(p^{e_{\delta,t}}-1)(1-\delta)t,~ 
(p^{e_{\delta,t}}-1)(1-\delta)\lambda_t, \in \mathbb Z,
$ 
\item $(p^{e_{\delta,t}}-1)(1-\delta)E_t$ and $(p^{e_{\delta,t}}-1)(K_X+\Delta)$ 
are integral, 
\item $(p^{e_{\delta,t}}-1)B_\delta$ is Cartier, and 
\item 
$
(p^{e_{\delta,t}}-1)(1-\delta)(B|_U+\lambda_tD|_U-th^*K_V) \sim_{\mathbb Z} (p^{e_{\delta,t}}-1)(1-\delta)E_t.
$ 
\end{itemize}
Note that the same conditions hold for $e\in \mathbb Z_{>0}$ with $e_{\delta,t}|e$. 
By Proposition~\ref{prop:B+E}, we may assume that 
there is a $d_{\delta,t}\in \mathbb Z_{>0}$ such that 
$$
h_*\mathcal O_U\big( 
(p^e-1)(B_\delta|_U +d_{\delta,t}D|_U+(1-\delta)E_t)
\big) \otimes \mathcal O_V(H)
$$
is generically globally generated for each $e\in\mathbb Z_{>0}$ 
with $e_{\delta,t}|e$. 

\ul{Goal of this step:} we prove that 
$$
\mathcal F_{e, \delta, t}
:=h_*\mathcal O_U\big(
(1-p^e)(K_U+\Delta|_U) +p^eB|_U +\mu_eD|_U
+h^* \left( -\nu_et K_V +H \right)
\big) 
$$
is generically generated by its global sections 
for $e\gg0$ with $e_{\delta,t}|e$, where 
$$
\mu_e:=p^e(1 +d_{\delta,t} +\lambda_t) 
\quad \textup{and} \quad 
\nu_e:=(p^e-1)(1-\delta). 
$$ 
Note that $\nu_e t\in \mathbb Z$. 
We now have 
\begin{align*} 
& (1-p^e)(K_U+\Delta|_U) +p^e B|_U +\mu_eD|_U
+ h^* \left( -\nu_et K_V +H \right)
\\ & \ge (1-p^e)(K_U+\Delta|_U) +p^e B|_U +(p^e-1)(1+d_{\delta,t}+(1-\delta)\lambda_t) D|_U
\\ & \hspace{300pt} + h^* \left( -\nu_et K_V +H \right)
\\ & = B|_U +(p^e-1)(\delta B|_U +D|_U-(K_U+\Delta|_U)) 
+(p^e-1)d_{\delta,t}D|_U
\\ & \hspace{230pt} +\nu_e(B|_U +\lambda_tD|_U -th^*K_V) +h^* H
\\ & \hspace{270pt} \textup{\footnotesize{(note that $p^e=1+(p^e-1)\delta + \nu_e$)}}
\\ & \sim_{\mathbb Z} B|_U +(p^e-1)B_\delta|_U +(p^e-1)d_{\delta,t}D|_U 
+\nu_eE_t +h^* H
\\ & = B|_U +(p^e-1)(B_\delta|_U +d_{\delta,t}D|_U +(1-\delta)E_t) +h^* H, 
\end{align*}
so there is the injective morphism 
$$
\mathcal F'_{e,\delta,t}:=
h_*\mathcal O_U\big( 
B|_U +(p^e-1)(B_\delta|_U +d_{\delta,t}D|_U +(1-\delta)E_t) +h^*H)
\big) 
\hookrightarrow 
\mathcal F_{e, \delta,t}, 
$$
which is generically isomorphic as $D|_{X_\eta}=0$. 
It is enough to show that 
$\mathcal F'_{e,\delta,t}$ is generically globally generated. 
Consider the following morphisms:
\begin{align*}
& (f_* \mathcal O_X(B))|_V \otimes h_* \mathcal O_U\big((p^e-1)(B_\delta|_U +d_{\delta,t}D|_U +(1-\delta)E_t) \big) \otimes \mathcal O_V(H) 
\\ \cong & h_* \mathcal O_X(B|_U) \otimes h_* \mathcal O_U\big((p^e-1)(B_\delta|_U +d_{\delta,t}D|_U +(1-\delta)E_t) +h^* H \big) 
\\ \xrightarrow{\varphi} & h_* \mathcal O_U\big(B|_U +(p^e-1)(B_\delta|_U +d_{\delta,t}D|_U +(1-\delta)E_t) +h^* H \big)
=\mathcal F'_{e,\delta,t}. 
\end{align*}
Here, the first isomorphism follows from the projection formula, 
and $\varphi$ is the natural morphism. 
As mentioned above, 
$$
h_* \mathcal O_U\big( 
(p^e-1)(B_\delta|_U +d_{\delta,t}D|_U +(1-\delta)E_t) 
\big) \otimes \mathcal O_V(H) 
$$
is generically generated by its global sections. 
Also, 
$
(f_*\mathcal O_X(B))|_V
$
is globally generated by the choice of $B$. 
Therefore, it is enough to show that $\varphi$ is generically surjective. 
We now have 
\begin{align*}
& (p^e-1)(B_\delta|_U +d_{\delta,t}D|_U +(1-\delta)E_t) +h^*H 
\\ \sim_{\mathbb Z} & 
(p^e-1)\big(\delta B|_U +D|_U -(K_X+\Delta)|_U +d_{\delta,t}D|_U 
+(1-\delta)(B|_U+\lambda_t D|_U -th^*K_V)\big)  
\\ & \hspace{370pt} +h^*H 
\\ = &
(p^e-1)(B +D-(K_X+\Delta))|_U +(p^e-1)(d_{\delta,t}+(1-\delta)\lambda_t)D|_U 
  \\ & \hspace{240pt} +h^*(-(p^e-1)(1-\delta) tK_V +H). 
\end{align*}
Since $D-(K_X+\Delta)$ is nef and $D|_{X_\eta}=0$, 
the generic surjectivity of $\varphi$ follows from the choice of $B$. 
\end{step}
Replacing $e_{\delta,t}$ by its multiple, we may assume that for each integer $e>0$ with $e_{\delta,t}|e$, the sheaf $\mathcal F_{e,\delta,t}$ is generically generated by its global sections. 
\begin{step} \label{step:negative}
\ul{Setup:}
Let $W_e$ be the normalization of the reduced subscheme of $U \times_V V^e$, 
and introduce the following names for the induced morphisms:
%
\begin{equation*}
\xymatrix@C=70pt@R=60pt{
U^e \ar[rd]_{\gamma_{e-1}^1} \ar[drrr]^{c_{e,s}} \ar[rrd]_{\gamma_e} \ar@/_5pc/[ddrr]_{h^e} \ar@/^3pc/[drrrr] \\
&  W_{e-1}^1 \ar[r]_{\delta_e} \ar[dr]_{\alpha_{e-1}^1} & W_e \ar[d]_{\alpha_e } \ar@/^1.5pc/[rr]^{\beta_e} \ar[r]_{b_{e,s}} & W_{s} \ar[r]_{\beta_{s}} \ar[d]_{\alpha_{s}} & U \ar[d]^h \\
 & & V^e \ar@/_1.5pc/[rr]_{F_Z^e}  \ar[r]^{F_V^{e-s}} & V^{s} \ar[r]^{F^{s}_V} & V
}
\end{equation*}
Put $W_0:=U$. Let $r$ be the minimum non-negative integer such that $\alpha_r:W_r\to V^r$ is separable. 

\ul{Goal of this step:}  we show that $\gamma_r^*K_{W_r/U} \le 0$. Note that when $f$ is separable (case~\eqref{itm:weak_pos:separable_fiber}), then $r=0$ and $\gamma_r^*K_{W_r/U}=0$. Thus we may assume that $a_Y:Y\to A_Y$ is separable over its image (case~\eqref{itm:weak_pos:separable_over_image}). 
Set 
\begin{equation*}
\sH^e=\imSat\Big( \alpha_{e-1}^{1,*} \Omega_{V^e}  \to \Omega_{W_{e-1}^1} \Big)
\end{equation*}
where $\imSat$ denotes the saturated image. Then, by Lemma \ref{lem:foliation} and by the canonical bundle formula for purely inseparable morphisms of height 1, we have 
\begin{multline*}
\delta_e^* K_{W_e/W_{e-1}} = K_{W_{e-1}^1/W_{e-1}} + (p-1) \det \left( \factormiddle{\Omega_{W_{e-1}^1}}{\sH^e} \right) 
\\ = (1-p) K_{W_{e-1}^1} + (p-1) K_{W_{e-1}^1} + (1-p) \det \sH^e
= (1-p) \det \sH^e
\end{multline*}
by Lemma~\ref{lem:foliation}. Hence, we obtain that 
\begin{equation*}
\label{eq:canonical_W_e_over_W_e_minus_1}
\gamma^*_r K_{W_r/U} = \sum_{s=1}^r F_U^{r-s,*} \gamma^{1,*}_{s-1} (1-p) \det \sH^s. 
\end{equation*}
Since $Y\to a_Y(Y)$ is separable, the homomorphism 
\begin{equation*}
d(a_Y) : \sO_Y^{\oplus \dim A} \cong a_Y^* \Omega_A \to \Omega_Y
\end{equation*}
is generically surjective. It follows then that $\sH^e$ is generically globally generated, so $\det \sH^e \geq 0$, and hence $\gamma_r^*K_{W_r/U} \le 0$. 
\end{step}
\begin{step} \label{step:T}
\ul{Goal of the step:} we show that the set $T$ defined in Step \ref{step:F} is not bounded from above. 
For each $e\ge r$, the scheme $W_r\times_{V^r}V^e$ is reduced by the choice of $r$, so we have $K_{W_e} \le d_e^*K_{W_r\times_{V^r}V^e}$, where $d_e:W_e\to W_r\times_{V^r}V^e$ is the induced morphism.
Then 
\begin{align*}
K_{W_e/U}
\le d_e^*K_{W_r\times_{V^r}V^e/U} 
& = d_e^*K_{W_r\times_{V^r}V^e/W_r} +b_{e,r}^*K_{W_r/U}
\\ & = \alpha_e^*K_{V^e/V^r} +b_{e,r}^*K_{W_r/U} 
= \alpha_e^*(1-p^{e-r})K_V +b_{e,r}^*K_{W_r/U}, 
\end{align*}
and so by Step \ref{step:negative} we have
\begin{align*} \tag{$\ast$} \label{ineq:1}
\gamma_e^*K_{W_e/U}
\le h^{e,*}(1-p^{e-r})K_V +F_U^{e-r,*}\gamma_r^*K_{W_r/U}
\le h^{e,*}(1-p^{e-r})K_V. 
\end{align*}
Fix a $t\in T$, a $\delta \in \mathbb Z_{(p)} \cap (0,1)$ and an integer $e\gg r$ with $e_{\delta,t}|e$. 
Applying $\mathcal Hom((?),\mathcal O_{W_e})$ to the composite of 
$$
\sO_{W_e} 
\to \gamma_{e,*}\sO_{U^e} 
\hookrightarrow \gamma_{e,*}\sO_{U^e}((1-p^e)\Delta|_U), 
$$
we obtain the generically surjective morphism 
$$
\gamma_{e,*}\sO_{U^e}((1-p^e)\Delta|_U +K_{U^e/W_e}) \to \sO_{W_e}
$$
by the Grothendieck duality. 
Combining this and $\gamma_e^\sharp:\sO_{W_e}\to \gamma_{e,*}\sO_{U^e}$, we get the  non-zero homomorphism 
$$
\gamma_{e,*}\sO_{U^e}((1-p^e)\Delta|_U +K_{U^e/W_e}) 
\to \gamma_{e,*}\sO_{U^e}.
$$
Taking reflexive tensor product with $\sO_{W_e}(K_{W_e/U})$, we obtain 
$$
\gamma_{e,*}\sO_{U^e}((1-p^e)(K_U+\Delta|_U)) 
\to \gamma_{e,*}\sO_{U^e}(\gamma_e^*K_{W_e/U}) 
\to \gamma_{e,*}\sO_{U^e}(h^{e,*}(1-p^{e-r})K_Z), 
$$
where the second homomorphism uses \eqref{ineq:1}. Let $u_e$ denote the composite of the above homomorphisms. 
Put 
$$
M_e:=\beta_e^*(B|_U +(1+d_{\delta,t}+\lambda_t)D|_U) 
+\alpha_e^*\left( t(1-p^e)(1-\delta)K_V +H \right). 
$$
Applying the functor 
$
\alpha_{e,*} \Big( \left( 
(?) \otimes \mathcal O_{W_e}\left( M_e \right) 
\right)^{**} \Big)
$ 
to $u_e$, we get the homomorphism 
\begin{align*}
v_e: h^e_* \mathcal O_{U^e}\left( 
(1-p^e)(K_U+\Delta|_U) +\gamma_e^* M_e \right)
\to 
h^e_* \mathcal O_{U^e} \left(h^{e,*}(1-p^{e-r})K_Z + \gamma_e^* M_e \right)
\end{align*}
by the projection formula. 
We will show that $v_e\ne 0$ in Step~\ref{step:v}. 
We assume that $v_e\ne 0$. 
Since 
\begin{align*}
& (1-p^e)(K_U+\Delta|_U) +\gamma_e^*M_e 
\\ & = (1-p^e)(K_U+\Delta|_U) +p^e(B|_U +(1+d_{\delta,t}+\lambda_t)D|_U) 
\\ & \hspace{200pt} +h^{e,*}\left( t(1-p^e)(1-\delta)K_V +H \right)
\\ & = (1-p^e)(K_U+\Delta|_U) +p^eB|_U +\mu_eD|_U 
+h^{e,*}\left( -\nu_et K_V  +H \right), 
\end{align*}
we see that the source of $v_e$ is equal to $\mathcal F_{e,\delta,t}$, 
which is generically globally generated by Step~\ref{step:F}, 
and hence so is $\mathrm{Im}\left(v_e\right) \ne 0$. 
Therefore, we obtain that 
$$
H^0\left(U^e, h^{e,*}(1-p^{e-r})K_Z +\gamma_e^*M_e \right)
\ne 0, 
$$ 
so we find an effective Cartier divisor 
$E\sim_{\mathbb Z} h^{e,*}(1-p^{e-r})K_Z +\gamma_e^* M_e$. 
Since $h^* H \le B|_U$ by the choice of $B$, we have 
\begin{align*}
0 \le E & \sim_{\mathbb Z} h^{e,*}(1-p^{e-r})K_Z +\gamma_e^*M_e
\\ & = p^e(B|_U +(1+d_{\delta,t}+\lambda_t)D|_U) 
+h^{e,*} \big((t(1-p^e)(1-\delta) +1-p^{e-r})K_V +H \big) 
\\ & \le (p^e+1)B|_U +\mu_eD|_U +h^{e,*}\big((t(1-p^e)(1-\delta) +1-p^{e-r})K_V \big), 
\end{align*}
so we get that 
$$
B|_U + \left\lceil \frac{\mu_e}{p^e+1} \right\rceil D|_U 
+\frac{t(p^e-1)(1-\delta) +1-p^{e-r}}{p^e+1} h^* K_V
\sim_{\mathbb Z_{(p)}} E' \ge 0
$$
for a $\mathbb Z_{(p)}$-Cartier divisor $E'\ge0$, 
which means that 
$$
T \ni \frac{t(p^e-1)(1-\delta) +1-p^{e-r}}{p^e+1} 
\xrightarrow{e\to +\infty,~\delta\to 0} 
t+p^{-r}. 
$$
Hence, $T$ cannot be bounded from above. 
\end{step}
\begin{step} \label{step:v}
\ul{Goal of the step:} we prove that $v_e\ne 0$ when $e\gg0$. 
Since $D|_{X_\eta}=0$, we can ignore $D$, so we may assume that $D=0$. 
Note that $v_e\ne 0$ if and only if $v_e\ne 0$ at the generic point, as the target is torsion-free. 
By the same reason, we assume that $K_V=H=0$. Then $M_e=\beta_e^*B|_U$. 
Now, $v_e$ is the push-forward by $\alpha_e$ of the morphism 
$$
u'_e: \gamma_{e,*}\sO_{U^e}((1-p^e)(K_U+\Delta|_U) +p^eB|_U) 
\to \gamma_{e,*}\sO_{U^e}(p^eB|_U). 
$$
We show that $H^0(u'_e)$ is non-zero, which means that $v_e=\alpha_{e,*}u'_e\ne 0$. To show that, we prove that the source of $u'_e$ is globally generated. Note that $u'_e\ne 0$ by the construction. 
We first consider the case when $U=X$ (i.e., $V=Y$).
Recall that $B=(\dim X+1)B'$ for a very ample Cartier divisor $B'$ on $X$. 
Since $\gamma_e$ is finite, for each $1\le i \le \dim X$, we have 
\begin{align*}
& H^i\big(W_e, \gamma_{e,*}\sO_{X^e}((1-p^e)(K_X+\Delta) +p^eB)\otimes \sO_{W_e}(-i\beta_e^*B')\big)
\\ \cong & H^i\big(W_e, \gamma_{e,*}\sO_{X^e}((1-p^e)(K_X+\Delta) +p^e(\dim X+1-i)B')\big)
\\ \cong & H^i\big(X_e, \sO_{X^e}((1-p^e)(K_X+\Delta) +p^e(\dim X +1 -i)B')\big)
\\ = & 0
\end{align*}
by Fujita's vanishing theorem if $e\gg0$. 
Hence, our claim follows from the Castelnuovo--Mumford regularity~\cite[Theorem~1.8.5]{Laz04I}. 
When $U\ne X$, we can prove the claim by considering the same argument with replacing $W_e$ by the normalization of the reduced part of $X\times_Y Y^e$, which is projective. 
\end{step}
\begin{step} \label{step:last} 
\ul{Goal of the step:} we conclude the proof of the theorem.  
Put 
$$
T_0:=\left\{t\in\mathbb Q \middle| 
\textup{$-h^*K_V|_{U_0} +tB|_{U_0} \sim_{\mathbb Q} E$ for a $\mathbb Q$-Weil divisor $E\ge 0$ on $U_0$}. 
\right\}
$$
Then the assertion is equivalent to that 
$\mathbb Q_{>0} \subseteq T_0$. 
By Step~\ref{step:T}, one can easily check that 
$\mathbb Z_{(p)}\cap [0,\infty) \subseteq T$. 
By definition, if $0<t\in T$, then $t^{-1}\in T_0$. 
Hence, we see that $\mathbb Q_{>0}\subseteq T_0$. 
\end{step}
\end{proof}
\section{Proof of Theorem~\ref{thm:main}}
Next, we prove the main theorem of this paper.
\begin{proof}[Proof of Theorem~\ref{thm:main}]
First, we prove \eqref{itm:main:Albanese_surjective}. Let $Y$ be the normalization of $a(X)$ and let $f:X\to Y$ be the induced morphism. 
Put $D=0$. 
Let $V$, $U$, $h:U\to V$ be as in Theorem~\ref{thm:weak pos}. 
Then by Theorem~\ref{thm:weak pos}, we see that 
$\mathcal O_U(-h^*K_V)$ is weakly positive. 
By \cite[Lemma~2.4~(2)]{EG19}, 
we obtain that $\omega_V^{-1}$ is weakly positive, 
which is equivalent to that $\omega_Y^*$ is weakly positive.
Hence, Proposition~\ref{prop:main} tells us that $Y$ is an abelian variety, 
i.e. $a$ is surjective. 

We prove \eqref{itm:main:Stein}. Let $X\xrightarrow{f'}Y'\xrightarrow{g}A$ be the Stein factorization of $a$. Let $Y$ be the normalization of $A$ in the separable closure of $K(Y')/K(A)$. Let $D$, $V$, $U$ and $h:U\to V$ be as in the proof of \eqref{itm:main:Albanese_surjective}. Then by the same argument as that of the proof of \eqref{itm:main:Albanese_surjective}, we see that $Y$ is an abelian variety, so $Y=A$ and $g$ is purely inseparable. 
By an argument similar to that of the proof of \eqref{itm:main:Albanese_surjective}, we can prove \eqref{itm:main:fiber_space}.

Next, we show \eqref{itm:main:a_exc}. 
Suppose that there is a prime divisor $E$ on $X$ such that 
the codimension of $a(\mathrm{Supp}(E))$ is at least two. 
Let $A'$ be the normalization of the flattening of $a$. 
Let $X'$ be the normalization of the main component of $X\times_A A'$. 
We have the following commutative diagram: 
$$
\xymatrix{ 
	X' \ar[r]^-\sigma \ar[d]_-{a'} & X \ar[d]^-a \\
	A' \ar[r]_-\tau & A.
}
$$
Since $A$ is smooth and $K_A=0$, we have $K_{A'}\ge0$ and 
${a'}^*K_{A'} \ge \sigma^{-1}_*E$. 
Note that, since $a'$ is equi-dimensional,  
the pullback by $a'$ of a Weil divisor on $A'$ can be defined. 
Thus $\sigma_*{a'}^*K_{A'} \ge E$. 
We prove that $\mathcal O_X\left(-\sigma_*{a'}^*K_{A'}\right)$ is 
weakly positive. If this holds, then $\mathcal O_X(-E)$ is also weakly positive, which means that $E=0$ by \cite[Lemma~4.4]{Eji19w}, a contradiction. 
Since $K_X+\Delta$ is $\mathbb Z_{(p)}$-Cartier and $\sigma$ is isomorphic 
on the generic fiber $X'_\eta$ of $a'$, 
there are effective $\mathbb Z_{(p)}$-Weil divisor 
$\Delta'$ and $D$ on $X'$ such that 
$$
K_{X'} + \Delta' = \sigma^*(K_X+\Delta) +D, 
$$
that $D$ is $\sigma$-exceptional, and that $D|_{X'_\eta}=0$. 
Then 
\begin{itemize}
\item $\left(X'_\eta,\Delta'|_{X'_\eta}\right)$ is strongly $F$-regular, and 
\item $D-(K_{X'}+\Delta') = -\sigma^*(K_X+\Delta)$ is 
a nef $\mathbb Z_{(p)}$-Cartier divisor. 
\end{itemize}
Let $V'$ be a regular open subset of $A'$ such that 
$\mathrm{codim}(A'\setminus V') \ge 2$ 
and $h':=a'|_{U'}:U'\to V'$ is flat, where $U':={a'}^{-1}(V')$. 
Put $U'_0:=U'\setminus \mathrm{Supp}(D)$. 
We have the following commutative diagram: 
$$
\xymatrix{
	U'_0 \ar@{^(->}[r] & U' \ar@{^(->}[r] \ar[d]_-{h'} & X' \ar[d]^-{a'} \\
			  & V' \ar@{^(->}[r] & A'
}
$$
By Theorem~\ref{thm:weak pos}, we obtain that 
$\mathcal O_{U'_0}\left(-{h'}^*K_{V'}|_{U'_0}\right)$ is weakly positive. 
Since $a'$ is equi-dimensional and $D$ is $\sigma$-exceptional divisor, 
we see that $\sigma(X'\setminus U'_0)$ is of codimension at least two. Set $U=X 
\setminus \sigma(X'\setminus U'_0)$. Then $U \subseteq X$  is an open subset such that $U^\flat_0:=\sigma^{-1}U \subseteq U_0'$,
$\mathrm{codim}(X\setminus U)\ge 2$ and $\sigma^\flat:=\sigma|_{U^\flat_0}: U^\flat_0 \to U$ is 
an isomorphism. In particular,
 $\mathcal O_{U^\flat_0}\left(-{h'}^*K_{V'}|_{U^\flat_0}\right)$ 
is weakly positive, and hence
$$
\sigma^\flat_*\mathcal O_{U^\flat_0}\left(-{h'}^*K_{V'}|_{U^\flat_0}\right)
\cong \mathcal O_U \left(-\sigma^\flat_*{h'}^*K_{V'}|_{U^\flat_0}\right)
\cong \mathcal O_X\left(-\sigma_*{a'}^*K_{A'}\right) |_U
$$
is also weakly positive. This means that 
$
\mathcal O_X\left(-\sigma_*{a'}^*K_{A'}\right) 
$
is weakly positive, since the weak positivity is determined on any open subset 
whose complement has codimension at least two. 
\end{proof}
\begin{rem} \label{rem:alb}
Theorem~\ref{thm:main} has already been proved under the assumption that 
the pair $(X_{\overline\eta},\Delta|_{X_{\overline\eta}})$ is $F$-pure, 
where $X_{\overline\eta}$ is the geometric generic fiber of $X \to \mathrm{Im}(a)$ (\cite[Theorem~1.1]{Eji19w}), 
but this assumption does not always hold.  
For instance, if $X$ is a quasi-hyperelliptic surface (cf. \cite[\S 2]{BM3}), 
then $-K_X$ is nef and $a$ forms an algebraic fiber space, 
but $X_{\overline\eta}$ has a cusp, which is not $F$-pure (cf. \cite[Theorem~1.1]{GW77})

Additionally, $X_{\overline\eta}$ can even be non-reduced, examples of which can be found in \cite[Section 14]{PZ19}.
\end{rem}
\begin{proof}[Proof of Corollary \ref{cor:characterization_abelian}]
Point \eqref{itm:characterization_abelian:inequality} follows directly from Theorem \ref{thm:main_smooth}, taking into account that $2(\dim A) = b_1$ \cite[Prop 2.6]{PZ19}.

For point \eqref{itm:characterization_abelian:characterization},  if $X$ is an abelian variety, then the Albanese morphism is the identity, and  hence $2(\dim X) = 2 (\dim A)  = b_1$. Therefore, it is enough to prove the opposite direction, for which assume that $b_1= 2 \dim X$. By using $2(\dim A) = b_1$ again, we obtain that $\dim X = \dim A$, and then $a: X \to A$ is generically finite  by Theorem \ref{thm:main_smooth}. However, then Theorem \ref{prop:main} concludes that $X$ is an abelian variety. 
\end{proof}
\begin{proof}[Proof of Corollary~\ref{cor:surface}]
Let $X\xrightarrow{f} Y \to A$ be the Stein factorization of $a$. 
If $\dim Y=0$, then $Y=A$, so the claim follows. 
If $\dim Y=1$, then $f:X\to Y$ is separable, so the assertion follows from 
Theorem~\ref{thm:main}. 
If $\dim Y=2$, then $X$ is of maximal Albanese dimension, 
so Proposition~\ref{prop:main} completes the proof. 
\end{proof}
\begin{proof}[Proof of Corollary~\ref{cor:K=0}]
Let $X\xrightarrow{f}Y\to A$ be the Stein factorization of $a$. 
Since $p \ge 5$ and $K_X$ is numerically trivial, by \cite[Corollary~1.8]{PW22}, 
we see that the general fiber of $f$ is smooth.  In particular, $f$ is separable. 
Hence, the assertion follows from Theorem~\ref{thm:main}. 
\end{proof}
\section{Algebraic fiber spaces over elliptic curves}
\begin{thm}[\textup{\cite[Theorem~3.1]{PZ19}}] \label{thm:nefness}
Let $(X,\Delta)$ be a projective pair
such that $K_X+\Delta$ is $\mathbb Z_{(p)}$-Cartier, let
$f:X\to Y$ be a surjective morphism to a normal projective variety $Y$, and let
Let $L$ be a $\mathbb Q$-Cartier divisor on $X$. 
Suppose that 
\begin{itemize}
\item $(X,\Delta)$ is $F$-pure, 
\item $L$ is nef and $K_X+\Delta+L$ is $f$-nef, and 
\item $Y$ is of maximal Albanese dimension. 	
\end{itemize}
Then $K_X+\Delta+L$ is pseudo-effective. 
Moreover, if $\dim Y=1$, then $K_X+\Delta+L$ is nef. 
\end{thm}
\begin{proof}
Let $A$ be an ample $\mathbb Z_{(p)}$-Cartier divisor on $X$. 
We show that $K_X+\Delta+L+A$ is pseudo-effective. 
Since $L+A$ is ample, by \cite[Corollary~6.10]{SW13}, 
there is an effective $\mathbb Z_{(p)}$-divisor 
$\Gamma \sim_{\mathbb Z_{(p)}} L+A$ 
such that $(X,\Delta+\Gamma)$ is $F$-pure. 
Put $\mathcal F:=f_*\mathcal O_X(m(K_X+\Delta+\Gamma))$ 
for an integer $m$ large and divisible enough. 
Since 
$$
K_X+\Delta+\Gamma \equiv K_X+\Delta+L+A
$$ 
is $f$-ample, we can apply Proposition~\ref{prop:genglgen}.\eqref{itm:genglgen:diminished} and obtain that 
$ \mathbb B_-(\mathcal F) \ne Y$. 
Let $H$ be an ample Cartier divisor on $Y$. 
We use the notation of \cite[\S 4]{Eji19d}. 
We then have the following:
\begin{align*}
\mathbb B_-(K_X+\Delta+\Gamma)
= & \mathbb B_-(m(K_X+\Delta+\Gamma))
\\ \subseteq & \mathbb B_-^{f^*H}(m(K_X+\Delta+\Gamma))
\hspace{30pt}\textup{\footnotesize{by \cite[Corollary~4.2]{Eji19d}}}
\\ \subseteq & \mathbb B_-^{f^*H}(f^*\mathcal F)
\hspace{90pt}\textup{\footnotesize{by $f^*\mathcal F \twoheadrightarrow \mathcal O_X(m(K_X+\Delta+\Gamma))$}}
\\ \subseteq & f^{-1}(\mathbb B_-^H(\mathcal F))
\ne X.
\end{align*}
Hence $K_X+\Delta+\Gamma$ is pseudo-effective. 

Next, we show that $K_X+\Delta+\Gamma$ is nef when $\dim Y=1$. 
Let $C$ be a projective curve in $X$. 
Since $K_X+\Delta+\Gamma$ is $f$-ample, 
we only need to consider the case when $f(C)=Y$. 
In this case, $C$ intersects $X\setminus B_-(K_X+\Delta+\Gamma)$ 
by the above argument, so $(K_X+\Delta+\Gamma \cdot C)\ge 0$. 
Hence $K_X+\Delta+\Gamma$ is nef, which means that so is $K_X+\Delta+L$. 
\end{proof}
\begin{notation} \label{notation:elliptic}
Let $(X, \Delta)$ be a projective pair of dimension $d+1$ 
such that 
\begin{itemize}
\item $(X,\Delta)$ is strongly $F$-regular, and 
\item $-(K_X+\Delta)$ is a nef $\mathbb Z_{(p)}$-Cartier divisor. 
\end{itemize}
Let $f:X\to T$ be a surjective morphism to an elliptic curve $T$, 
let $\tilde L$ be a very ample Cartier divisor on $X$ such that 
$R^if_*\mathcal O_X(\tilde L)=0$ for each $i>0$, and
let $G$ be a general fiber of $f$. 
Put $L:=n\tilde L -mG$, where $n:=(d+1)L^d\cdot G$ and $m:=\tilde L^{d+1}$. 
\end{notation}
\begin{thm}[\textup{cf. \cite[Theorem~5.4]{PZ19}}]
In the situation of Notation~\ref{notation:elliptic}, 
$L$ is nef. 
\end{thm}
\begin{proof}
Let $A$ be an ample $\mathbb Z_{(p)}$-Cartier divisor on $T$.
It is enough to prove that $L+f^*A$ is nef. 
By the argument in \cite[Proof of Theorem~5.4]{PZ19}, 
there is an effective $\mathbb Z_{(p)}$-Cartier divisor $\Gamma$ on $X$ 
such that $\Gamma\sim_{\mathbb Z_{(p)}} L+f^*A$. 
Take an $\varepsilon\in\mathbb Z_{(p)}$ so that $(X,\Delta+\varepsilon\Gamma)$
is $F$-pure. Since 
$$
\varepsilon(L+f^*A)
\equiv \varepsilon \Gamma
\equiv K_X+(\Delta+\varepsilon\Gamma) +(-(K_X+\Delta))
$$
and $-(K_X+\Delta)$ is nef, we see from Theorem~\ref{thm:nefness} that 
$L+f^*A$ is nef. 
\end{proof}
\begin{thm}[\textup{cf. \cite[Theorem~5.8]{PZ19}}] \label{thm:anti-nef}
In the situation of Notation~\ref{notation:elliptic}, 
for any integer $m\ge 1$, $(f_*\mathcal O_X(mL))^*$ is nef. 
\end{thm}
\begin{proof}
Suppose that $(f_*\mathcal O_X(mL))^*$ is not nef. 
Then by the classification of vector bundles on an elliptic curve~\cite{Oda71}, 
$f_*\mathcal O_X(mL)$ has an indecomposable vector bundle $\mathcal E$ 
with $\deg \mathcal E >0$ as a direct summand. 
Let $\pi:T'\to T$ be an \'etale cover whose degree is larger than $\mathrm{rank}\,\mathcal E$. 
Replacing $f:X\to T$ by $f_{T'}:X_T'\to T'$, 
we may assume that $\deg \mathcal E > \mathrm{rank}\,\mathcal E$. 
Note that $\pi^*f_*\mathcal O_X(L)\cong (f_{T'})_*\mathcal O_{X_{T'}}(L_{T'})$,
since $\pi$ is flat. 
Take a closed point $t\in T$. 
Then $\deg \mathcal E(-t) =\deg \mathcal E -\mathrm{rank}\,\mathcal E >0$, 
so we have 
\begin{align*}
0 \ne H^0(T, \mathcal E(-t)) 
& \subseteq H^0\big(T, (f_*\mathcal O_X(mL)) (-t) \big) 
\\ & = H^0(T, f_*\mathcal O_X(mL-G)) 
= H^0(X, \mathcal O_X(mL-G)),  
\end{align*}
where $G$ is a fiber of $f$ over $t$. 
Hence, we can find a $\Gamma\in |mL-G|$. 
Take $\varepsilon \in \mathbb Z_{(p)}$ so that 
$(X,\Delta+\varepsilon\Gamma)$ is strongly $F$-regular. 
Since $\Gamma$ is $f$-ample and 
$$
\varepsilon \Gamma 
\equiv K_X+(\Delta+\varepsilon\Gamma) + (-(K_X+\Delta)), 
$$
we can apply Theorem~\ref{thm:nefness} and obtain that $\Gamma$ is nef. 
However, we have 
$$
\Gamma^{d+1}
=(mL-G)^{d+1}
=-m^dL^d\cdot G <0,
$$
which is a contradiction. Thus, $(f_*\mathcal O_X(mL))^*$ is nef. 
\end{proof}
Recall that a vector bundle $\mathcal E$ is called numerically flat 
if $\mathcal E$ and $\mathcal E^*$ is nef. 
\begin{thm}[\textup{cf. \cite[Theorem~5.10]{PZ19}}] \label{thm:num flat}
In the situation of Notation~\ref{notation:elliptic}, 
there exists an integer $m_0\ge 1$ such that 
$f_*\mathcal O_X(mL)$ is numerically flat for each $m\ge m_0$. 
\end{thm}
\begin{proof}
Since $(X,\Delta)$ is $F$-pure and $L$ is $f$-ample, 
there is an integer $m_0\ge 1$ such that 
$$
S^0f_*(\sigma(X,\Delta)\otimes\mathcal O_X(mL)) 
=f_*\mathcal O_X(mL)
$$
for each $m\ge m_0$ 
(for the definition of $S^0f_*(\sigma(X,\Delta)\otimes\mathcal O_X(mL))$, 
see \cite[Definition~2.14]{HX15}). 
Fix an integer $m\ge m_0$. Put $\mathcal E:=f_*\mathcal O_X(mL).$
We prove that $\mathcal E$ is nef. 
Let $n\ge 1$ be an integer with $p\nmid n$. 
Let $\pi:T\to T$ be the morphism that sends $t$ to $nt$. 
For each integer $l\ge 1$, 
let $f_l:X_l\to T$ denote the morphism obtained by the base-change 
of $f:X\to T$ by $\pi^l:T\to T$. 
Since $\pi$ is \'etale, we can easily check that 
\begin{align*}
{\pi^l}^*\mathcal E
& = {\pi^l}^*S^0f_*(\sigma(X,\Delta)\otimes \mathcal O_X(mL))
\\ & \cong S^0{f_l}_*(\sigma(X_l,\Delta_l)\otimes \mathcal O_{X_l}(mL_l)), 
\end{align*}
where $\Delta_l$ and $L_l$ are the pullback of $\Delta$ and $L$ to $X_l$, 
respectively. 
Let $\mathcal A$ be an ample line bundle on $T$ with $|\mathcal A|$ free. 
Since $mL_l -(K_{X_l}+\Delta_l)$ is nef and $f_l$-ample, 
where $\Delta_l$ is the pullback of $\Delta$ to $X_l$, 
we see from \cite[Theorem~1.2]{Eji22b} that 
$
{\pi^l}^*\mathcal E \otimes \mathcal A^2
$
is globally generated. 
Suppose that $\mathcal E$ is not nef. 
Then by the classification of vector bundles on an elliptic curve, 
$\mathcal E$ has a direct summand $\mathcal G$ of negative degree. 
Then the degree of ${\pi^l}^*\mathcal G \otimes \mathcal A^2$ 
is negative if $l\gg0$,
which contradicts the global generation of 
${\pi^l}^*\mathcal E \otimes \mathcal A^2$. 
\end{proof}

\begin{prop}
\label{prop:trivialization_F_p}
In the situation of Notation \ref{notation:elliptic}, assume also that $k= \overline{\mathbb F}_p$. Then, there exists a base-change $T' \to T$ from an elliptic curve $T'$ and a polarized isomorphism of pairs over $T'$: 
\begin{equation*}
(X, \Delta; L) \times_T T' \cong (F, \Delta|_F; L|_F) \times_k T',
\end{equation*}
 where $F$ is any fiber of $f$. Note that the divisors in this isomorphism, except $\Delta$ before the base-change, are all $\mathbb Z_p$-AC  {\large(}$\mathbb Z_p$-almost Cartier{\large)} divisors in the sense of \cite{Har94}.
\end{prop}
\begin{proof}
Let $L$ be as in Notation~\ref{notation:elliptic}. 
Then $f_*\mathcal O_X(mL)$ is numerically flat for each $m\gg0$. 
We may assume that so is $f_*\mathcal O_X(L)$, replacing $L$ by $mL$. 
By the classification of vector bundles on an elliptic curve~\cite{Oda71}, 
we see that 
$$
f_*\mathcal O_X(L)
\cong \bigoplus_{i=1}^n \mathcal E_{r_i,0} \otimes \mathcal L_i, 
$$
where each $\mathcal E_{r_i,0}$ is an irreducible vector bundle of rank $r_i$ with 
non-zero global sections 
and each $\mathcal L_i$ is a line bundle of degree zero. 
Since $k=\overline{\mathbb F_p}$, each $\mathcal L_i$ is a torsion line bundle, 
so there is a finite morphism $\pi:T'\to T$ from an elliptic curve $T'$ such that 
$$
{f_{T'}}_* \mathcal O_{X_{T'}}(L_{T'})
\cong \pi^*f_*\mathcal O_X(L) 
\cong \bigoplus \mathcal O_{T'}.
$$  
Then the multiplication map 
$$
\bigoplus \mathcal O_{T'}
\cong S^m({f_{T'}}_*\mathcal O_{X_{T'}}(L_{T'}))
\twoheadrightarrow {f_{T'}}_*\mathcal O_{X_{T'}}(mL_{T'})
$$
shows that ${f_{T'}}_*\mathcal O_{X_{T'}}(mL_{T'})$ is globally generated. 
Combining this with the numerical flatness of 
${f_{T'}}_*\mathcal O_{X_{T'}}(mL_{T'})$, 
one can easily see that it is also a trivial vector bundle. 
Hence, the relative section ring 
$\bigoplus_{m\ge 0} {f_{T'}}_*\mathcal O_{X_{T'}}(mL)$
comes from $\mathrm{Spec}\,k$, 
which means that the morphism $f_{T'}:X_{T'}\to T'$ splits, 
i.e., $X_{T'} \cong G\times_k T'$, where $G$ is a fiber of $f$. 
Let $\Delta_{T'}$ be the pullback of $\Delta$ to $G\times_k T'$. 
Note that since $G$ satisfies $S_2$ and $G_1$ and 
$\pi_X:G\times_k T' \to X$ is flat, 
we can define $\Delta_{T'}$ as an effective $\mathbb Z_{(p)}$-AC divisor. 
We show that $\Delta_{T'} = \mathrm{pr}_1^*(\Delta|_G)$. 
Since 
$$
\pi_X^* K_X
= \pi_X^* K_{X/T}
\sim K_{G\times_k T'/T'}, 
$$
we have that $-\pi_X^*(K_X+\Delta) \sim -(K_{G\times_k T'/T'} +\Delta_{T'})$
is nef, so the claim follows from Lemma~\ref{lem:product divisor}.
\end{proof}
\begin{lem}[\textup{cf. \cite[Lemma~8.4]{PZ19}}] \label{lem:product divisor}
Let $G$ be a projective variety satisfying $S_2$ and $G_1$. 
Let $T$ be a normal projective variety. 
Let $\Delta$ be an effective $\mathbb Z_{(p)}$-AC divisor on $G\times_k T$. 
Suppose that $-(K_{G\times_k T/T} +\Delta)$ is a nef $\mathbb Z_{(p)}$-Cartier. 
Then $\Delta=\mathrm{pr}_1^*(\Delta|_{G\times_k t})$ for a closed point $t\in T$. 
\end{lem}
\begin{proof}
Let $h:G\times_k T \to G$ denote the first projection $\mathrm{pr}_1$.
By relaxing the projectivity of $G$ to quasi-projectivity, and by replacing the nefness assumption by $h$-nefness, we may shrink $G$ to its Gorenstein locus. With other words, we may assume that $G$ is Gorenstein.
Then $K_{G\times_k T/T} \sim h^* K_G$ is a Cartier divisor and it is $h$-nef. 
By the assumption, $-h^*K_G -\Delta$ is an $h$-nef $\mathbb Z_{(p)}$-Cartier divisor, 
and hence so is $-\Delta$. 
Therefore, for any projective curve $C\subseteq G\times_k T$ contracted by $h$, 
if $C$ meets $\mathrm{Supp}(\Delta)$, then $C\subseteq\mathrm{Supp}(\Delta)$.
This means that $h^{-1}(h(\mathrm{Supp}(\Delta))) =\mathrm{Supp}(\Delta)$. 
Thus the claim holds. 
\end{proof} 
\begin{proof}[Proof of Theorem~\ref{thm:isotrivial}]
Since $a(X)$ is a curve, the first part of the Stein factorization of $a:X\to A$ is separable, so Theorem~\ref{thm:main} implies that $a:X\to A$ forms an algebraic fiber space and $\dim A=1$. 
Put $T:=A$. 
Case~\eqref{itm:isotrivial:finite_field} is shown in Proposition~\ref{prop:trivialization_F_p}. Hence we only need to show that all fibers are isomorphic, and we may assume that $-(K_X + \Delta)$ is semi-ample. Choose a point $Q \in T$, a finitely generated $\mathbb F_p$ algebra $A \subseteq k$ and models
\begin{equation*}
\xymatrix{
(X_A, \Delta_A) \ar[r]^{f_A} & T_A  \ar[r] & \Spec A \ar@/_1pc/[l]_{\sigma_A}
}
\end{equation*}  
of $f : X \to T \ni Q$ over $A$. That is, $(f_A)_k=f$, $(\Delta_A)_k=\Delta$ and $(\sigma_A)_k$ corresponds to $Q$. We also assume that  there is also a very ample line bundle $L_A $ on $X_A$ such that $(L_A)_k =L$, and also that
\begin{itemize}
\item  $f_A$ is flat,
\item $\Spec A$ is regular
\item all fibers of $f_A$ and hence also $X_A$ is normal,
\item  $K_{X_A} +\Delta_A$ is $\mathbb Z_p$-Cartier, 
\item $(F, \Delta_F)$ is strongly $F$-regular for all fibers $F$ of $f_A$, where $\Delta_F = \Delta_A|_F$ (this restriction is meaningful as all fibers are normal), and
\item $-K_F - \Delta_F$ is semi-ample for all fibers $F$ of $f_A$.
\end{itemize}
As the above are all open properties that hold at the generic fiber (as they hold over $k$), we are able to achieve such a model. Call $Q_A$ the image of $\sigma_A$, which is a section of $T_A \to \Spec A$, and set $F_A = (f_A)^{-1} Q_A$. We denote then by $\Delta_{F_A}$ and $L_{F_A}$ the respective restrictions on $F_A$. Consider now the following $\Isom$ scheme \cite[Section 7, Prop 7.8]{PZ19},
\begin{equation*}
I= \Isom_{T_A} \Big( \big(F_A, \Delta_{F_A}; L_{F_A}\big)  \times_{\sigma, A} T_A,  \big(X_A, \Delta_A; L_A\big) \Big).
\end{equation*}
By Proposition \ref{prop:trivialization_F_p} and by the functoriality of $\Isom$ \cite[Prop 7.8.(1)]{PZ19}, $I$ surjects on $T_R$ for any $R \in T_A\big(\overline{\mathbb F}_p \big)$. Hence it is surjective in general. Using again the functoriality of $\Isom$ \cite[Prop 7.8.(1)]{PZ19}, we obtain the statement of the theorem. 
\end{proof}
\section{Case when anti-canonical divisor is ample} \label{section:ample}
\begin{proof}[The proof of Theorem \ref{thm:ample}]
We first prove case~\eqref{itm:ample:ample}. 
Let $X\xrightarrow{f} Y\to A$ be the Stein factorization 
of the Albanese morphism $a:X\to A$ of $X$. 
Let $L$ be an ample $\mathbb Z_{(p)}$-Cartier divisor on $Y$ such that 
$-(K_X+\Delta) -f^*L$ is ample. Then by \cite[Corollary~6.10]{SW13}, 
there is an effective $\mathbb Z_{(p)}$-Cartier divisor 
$D\sim_{\mathbb Z_{(p)}} -(K_X+\Delta) -f^*L$ such that 
$(X,\Delta+D)$ is $F$-pure. 
Then by Proposition~\ref{prop:genglgen}, we see that 
$$
f_* \mathcal O_X(m(K_X+\Delta+D)) \otimes \mathcal O_Y(H)
\cong 
f_* \mathcal O_X(-mf^*L) \otimes \mathcal O_Y(H)
\cong 
\mathcal O_Y(H-mL)
$$
is generically globally generated for each $m$ large and divisible enough, 
where $H$ was chosen as in Proposition~\ref{prop:genglgen}.
Therefore, $-L$ is pseudo-effective, which means that $\dim Y=0$, 
and so $\dim A=0$.

Next, we show case~\eqref{itm:ample:nef_big}. In this case, we can find 
an effective $\mathbb Z_{(p)}$-Cartier divisor on $X$ such that 
$(X,\Delta+E)$ is $F$-pure and that $-(K_X+\Delta+E)$ is ample. 
Hence, the assertion follows from case~\eqref{itm:ample:ample}.
\end{proof}

\bibliographystyle{alpha}
\bibliography{ejirizsolt12.bib}
\end{document}